\documentclass[10pt]{article}
\date{11.2.2013} 
\usepackage{hyperref}
\usepackage{amssymb, amsmath}
\input{liemacs10.sty} 

\newcommand\bD{\Bbb{D}}

\newcommand\slim{{s-\lim}}
\newcommand{\bk}{{\bf k}}
\newcommand{\bx}{{\bf x}}

\renewcommand\mlabel{\label}

\begin{document} 

\title{Norm continuous unitary representations of \\
Lie algebras of smooth sections} 
\author{B. Janssens
\begin{footnote}{Department  Mathematik, FAU Erlangen-N\"urnberg, Cauerstrasse 11, 
91058 Erlangen, Germany, Email:mail@bjadres.nl}
  \end{footnote}
\begin{footnote}
{Supported by the Emerging Field Project ``Quantum Geometry'' of the 
FAU Erlangen-N\"urnberg.}
\end{footnote},  
K.-H. Neeb
\begin{footnote}{Department  Mathematik, FAU Erlangen-N\"urnberg, Cauerstrasse 11, 
91058 Erlangen, Germany, Email: karl-hermann.neeb@math.uni-erlangen.de} 
  \end{footnote}
\begin{footnote}
{Supported by DFG-grant NE 413/7-2, Schwerpunktprogramm ``Darstellungstheorie''.}
\end{footnote}
}

\maketitle

\begin{abstract}  
Let $\fK \to X$ be a smooth Lie algebra bundle whose typical fiber is the 
compact Lie algebra $\fk$. 
We give a complete description 
of the bounded (i.e.\ norm continuous) unitary representations 
of the Fr\'echet--Lie 
algebra $\Gamma(\fK)$ of all smooth sections of $\fK$, 
and of the LF-Lie algebra $\Gamma_c(\fK)$ of compactly supported smooth sections.
For $\Gamma(\fK)$, bounded unitary irreducible representations are finite tensor products of 
so-called evaluation representations, hence in particular finite-dimensional. 
For $\Gamma_{c}(\fK)$, bounded unitary irreducible (factor) representations are
possibly infinite tensor products of evaluation representations, 
which reduces the classification 
problem to results of Glimm and Powers on    
irreducible (factor) representations of UHF $C^*$-algebras.
The key part in our proof is the result that
every irreducible bounded unitary representation 
of a Lie algebra of the form $\fk \otimes_{\mathbb{R}} \cA_{\mathbb{R}}$,
where $\cA_{\mathbb{R}}$ is a unital real continuous inverse algebra, 
is a finite product of evaluation representations.
On the group level, our results cover in particular the bounded 
unitary representations of the identity component $\Gau(P)_0$ 
of the group of smooth gauge transformations of a principal fiber bundle $P \to X$ 
with compact base and structure group, and 
the groups $\SU_n(\cA)_0$ with $\cA$ an involutive commutative 
continuous inverse algebra.\\

\nin{\em Keywords:} bounded unitary representation, evaluation representation, 
Lie algebra bundle, Lie algebra of sections, UHF algebra    
\end{abstract}

\section*{Introduction} \mlabel{sec:0}

If $G$ is a topological group, then we call a unitary 
representation $\pi \: G \to \U(\cH)$ {\it norm continuous} or {\it bounded} 
if $\pi$ is continuous with respect to the norm topology on the 
unitary group~$\U(\cH)$. If $G$ is a 
Lie group\begin{footnote}{Here we are using the term ``Lie group'' 
for Lie groups modeled on locally convex spaces.}  
\end{footnote}
with an exponential function $\exp \: \g \to G$ which is a 
local diffeomorphism in $0$, i.e., $G$ is {\it locally exponential}, 
then basic Lie theory implies that the continuous homomorphisms 
$\alpha \: \g \to \fu(\cH)$, i.e., the {\it bounded unitary 
representations} of the Lie algebra $\g$, are in one-to-one 
correspondence with the bounded unitary representations 
of the simply connected covering $\tilde G_0$ of the identity 
component $G_0$ of $G$.
For $1$-connected, i.e., connected and simply connected, 
locally exponential Lie groups, bounded 
unitary representations can therefore be understood completely 
in terms of representations of their Lie algebras. 

Let $\fK \to X$ be a smooth Lie algebra bundle 
whose typical fiber is the 
finite-dimensional Lie algebra $\fk$. 
In the present paper we shall give a complete description 
of the bounded projective unitary representations of the Fr\'echet--Lie 
algebra $\Gamma(\fK)$ of all smooth sections of $\fK$ 
(endowed with the smooth compact open topology) and 
of the LF-Lie algebra $\Gamma_c(\fK)$ of compactly supported smooth sections, 
endowed with the locally convex direct limit topology with respect to the 
Fr\'echet--Lie algebras $\Gamma(\fK)_C := \{ s \in  \Gamma(\fK) \: \supp(s) \subeq C\}$, 
where $C \subeq X$ is a compact subset. 
If $X$ is compact, then both Lie algebras coincide, but if $X$ is non-compact, 
then their respective bounded representation theory is quite different. 

It will be shown that the problem reduces to the case where 
$\fk$ is compact semisimple, and the representations are linear rather than projective.

For every point $x \in X$ and unitary representation 
$(\rho,V)$ of the fiber $\fK_x \cong \fk$, we obtain an irreducible 
unitary representation 
$\pi_{x,\rho}(s) := \rho(s(x))$
of $\Gamma(\fK)$. We call these representations 
{\it evaluation representations}. Our central result is that 
every irreducible bounded unitary representation of $\Gamma(\fK)$ 
is a finite tensor product of irreducible evaluation representations in 
different points of $X$ and a one-dimensional representation 
(of the center). In particular, it is finite-dimensional. 

If $X$ is non-compact, then the bounded representation 
theory of the LF-Lie algebra  $\Gamma_c(\fK)$ is ``wild'' in the sense that there exist 
bounded factor representations of type II and III. Here our main result 
is a complete reduction of the classification of bounded irreducible 
representations to the classification of irreducible representations 
of UHF $C^*$-algebras. This correspondence is established as follows. 
Let $\bx \subeq X$ be a locally finite subset and 
let $(\rho_x, V_x)$ be
irreducible representations of 
$\cK_{x} \cong \fk$ corresponding to $x\in \bx$.
We then form the $C^*$-algebra 
\[ \cA_{\bx,\rho} := \widehat{\bigotimes}_{x\in \bx} B(V_x) \] 
for which the evaluation representations $\pi_{x,\rho_x}$ combine to a 
continuous morphism $\eta_{\bx,\rho} \: \Gamma_c(\fK) \to \cA_{\bx,\rho}$. 
We then show that 
a bounded irreducible (factor) representation 
of $\Gamma_c(\fK)$ 
is of the form 
$\pi = \beta \circ \eta_{\bx,\rho}$ for some irreducible (factor)
representation of $\beta$ of some $\cA_{\bx, \rho}$. Conversely, every 
irreducible (factor) representation $\beta$ of $\cA_{\bx,\rho}$ 
defines a bounded irreducible (factor) representation 
$\beta \circ \eta_{\bx,\rho}$ of $\Gamma_c(\fK)$. 
Since $\pi$ determines the set $\bx$ and the 
representations $\rho_x$ uniquely, 
we thus obtain a complete reduction of the classification 
of bounded irreducible (factor) representations of $\Gamma_c(\fK)$ 
to the corresponding problem for the UHF $C^*$-algebras 
$\cA_{\bx,\rho}$. These algebras have been classified by 
Glimm in \cite{Gli60}, where one also finds a characterization 
of their pure states. Even stronger results were obtained later by 
Powers in \cite{Po67}, where he shows that the automorphism group 
acts transitively on the set of pure states, so that every irreducible 
representation is a twist (by an automorphism) of an infinite 
tensor product of irreducible representations. 

The key part in our proof is the corresponding result for 
$X$ compact and the trivial Lie algebra bundle, i.e., for 
$\Gamma(\fK) \cong C^\infty(X,\fk) \cong \fk \otimes_\R C^\infty(X,\R)$. 
This is achieved by dealing with a substantially 
larger class of Lie algebras of the form 
$\fk_\cA := \fk \otimes_\R \cA_\R$, where $\cA_\R$ is a commutative 
unital continuous 
inverse algebra, i.e., a locally convex topological algebra 
with open unit group and continuous inversion map. 
Here our main result asserts that every irreducible 
bounded unitary representation of $\fk_\cA$ 
is a finite tensor product of {\it evaluation 
representations} $\pi_{\chi,\rho}(x \otimes a) = \chi(a)\rho(x)$, 
where $\rho \: \fk \to \fu(V_\rho)$ is an irreducible unitary 
representation and $\chi \: \cA_\R \to\R$ is a character. 
This generalizes a similar result 
obtained in \cite{NS11} for the special case $\cA_\R = C(X,\R)$,
where $X$ is compact. 
The main new point here is the passage from Banach algebras to 
general continuous inverse algebras $\cA$. This is achieved by a more 
direct approach that does not rely on holomorphic line bundles 
and Banach spaces of holomorphic sections. 

The methods we use to deal with the Lie algebras 
of the form $\fk_\cA$ rely heavily on the fact that $\cA_\R$ is unital, 
but it turns out that if $\cA_\R= C_0(X,\R)$, where 
$X$ is a locally compact space countable at infinity, 
then every bounded unitary representation 
of $\fk_\cA \cong C_0(X,\fk)$ has a unique extension to 
$C(X_\omega,\fk) \cong C_0(X,\fk) \rtimes \fk$, where $X_\omega$ is the one-point compactification of~$X$ (cf.\ Section~\ref{sec:6}). 
This extends our classification result to algebras 
of the form $C_0(X,\fk)$. 

On the group level our results cover in particular the bounded 
unitary representations of the identity component $\Gau(P)_0$ 
of the group of smooth gauge transformations of a $K$-principal bundle $P \to X$, 
where $K$ is a compact Lie group and $X$ is compact, the groups 
$\SU_n(\cA)_0$, where $\cA$ is an involutive commutative continuous inverse algebra, 
and, more generally, connected groups with Lie algebra $\fk_\cA$. 

The structure of the paper is as follows. 
In Section~\ref{sec:1} we start with a key ingredient of our classification 
for the algebras $\fk_\cA$, where $\cA$ is an 
involutive commutative continuous inverse algebra  $\cA$. 
Every function $\phi \: \cA \to \C$ which is multiplicative and polynomial 
of degree $N$ is a product of $N$ algebra homomorphisms 
$\chi_j \: \cA \to \C$ (Theorem~\ref{thm:2.5}). 
In Section~\ref{sec:2} we analyze the 
bounded unitary representations of $\fk_\cA \cong \fk \otimes_\R \cA_\R$, where 
$\fk$ is compact semisimple, $\cA_\R := \{ a \in \cA \: a^* = a\}$ and 
$\cA$ is an involutive commutative continuous inverse algebra. 
This is done by studying their complex linear extensions 
to $\g(\cA) = \g \otimes_\C \cA$, where 
$\g := \fk_\C$ is the complexification of $\fk$. 
Let $\ft \subeq \fk$ be maximal abelian, so that  
$\g^0 := \ft \otimes_\R \cA \subeq \g(\cA)$ is maximal abelian in 
$\g(\cA)$ and there is a natural triangular decomposition 
$\g(\cA) \cong \fg^+ \oplus \g^0 \oplus \fg^-$ determined by 
a positive system of roots $\Delta^+$ of the semisimple complex 
Lie algebra~$\g$. 
If $\pi \: \g(\cA) \to B(\cH)$ is a bounded irreducible 
representation, then the subspace 
$\cH^{\fg^-}$ is a one-dimensional representation of $\g^0$, hence 
determined by a linear functional $\lambda \: \g^0 \to \C$, called 
its {\it highest weight}. 
Two such representations are equivalent if and only 
if their highest weights are equal. Therefore it remains to determine 
the set of highest weights. This is achieved in Theorem~\ref{thm:ind-charac},  
which asserts that these are precisely the sums of functionals of the form 
$\lambda \otimes \chi$, where 
$\lambda$ is dominant integral for $(\fg, \ft_\C, \Delta^+)$ 
and $\chi \: \cA \to \C$ is an involutive algebra homomorphism. 
Here the functionals of the form $\lambda \otimes \chi$ correspond to 
evaluation representations 
$\pi_{\chi,\lambda}(x \otimes a) := \chi(a) \rho_\lambda(x)$, where 
$\rho_\lambda$ is the irreducible representation of $\fk_\C$ of highest weight 
$\lambda$. Part of the proof of Theorem~\ref{thm:ind-charac}, namely 
the $\fsl_2$-reduction needed to show that all highest weights are 
sums of such functionals, is carried out in Section~\ref{sec:3}. 
We obtain in particular that all irreducible 
bounded unitary representations of $\fk_\cA$ are finite 
tensor products of irreducible evaluation representations. 

In Section~\ref{sec:4} we finally turn to the Lie algebras 
$\Gamma(\fK)$ and $\Gamma_c(\fK)$. For $\Gamma_c(\fK)$ 
an important observation is that 
every factor representation is a tensor product 
$\pi = \pi_1 \otimes \pi_2$, where 
$\pi_1$ can be chosen such that it vanishes on the ideal 
$\Gamma(\fK)_C$ of sections supported in a given compact subset 
$C \subeq X$ 
and $\pi_2$ is a finite tensor product of irreducible evaluation representations. 
Eventually, this factorization provides the bridge between 
Lie algebra representations and the infinite tensor products of 
matrix algebras. 
In Section~\ref{sec:5} we show that assuming compactness of the fiber Lie algebra 
$\fk$ and linearity of the projective representations
does not lead to loss of generality
in the classification of irreducible 
bounded projective  
unitary representations of $\Gamma(\fK)$. 
In Section~\ref{sec:6} we first observe that,
for $\fk_\cA$ with $\cA$ 
the non-unital Banach algebra 
$\cA = \ell^1(\N,\C)$, infinite tensor products of evaluation 
representations lead  to infinite-dimensional 
irreducible bounded unitary representations and to a wild bounded 
representation theory. This is in contrast to the fact 
that, for a locally compact space countable at infinity, 
all irreducible bounded representations of $C_0(X,\fk) \cong \fk \otimes_\R C_0(X,\R)$ 
are finite tensor products of evaluation representations, even though 
$\cA = C_0(X)$ is not unital. If $\cA$ is a Banach algebra with 
$\ell^1(\N,\C) \subeq \cA \subeq C_0(\N,\C)$, then the representation 
theory of $\fk_\cA$ becomes quite involved because some 
bounded representations of $\ell^1(\N,\fk)$ only extend to unbounded 
projective representations of $\fk_\cA$. 
We plan to pursue these aspects in the forthcoming paper \cite{JN13}. 
In Subsection~\ref{subsec:6.3} we assume that 
$X$ is compact with boundary and in this context we characterize those 
representations of $\Gamma_c(X^\circ)$ extending to the Banach--Lie algebra 
of $C^k$-section whose $k$-jet vanishes on the boundary $\partial X$.  \\

Finite-dimensional representations of Lie algebras of the form 
$\g \otimes \cA$, where $\g$ is semisimple and $\cA$ a unital commutative algebra, 
are presently under active investigation from an algebraic point of view. 
More generally, one studies {\it equivariant map algebras}, i.e., 
Lie algebras of the form 
$(\g \otimes \cA)^\Gamma$, where 
$\cA$ is the algebra of regular functions on an affine variety 
and $\Gamma$ a finite group acting on $\g$ and $\cA$. 
The irreducible finite-dimensional representations of equivariant 
map algebras have recently been 
classified by Neher, Savage and Senesi (\cite{NSS12})\label{NSSref}. 
Their main result asserts that they are finite tensor products of 
evaluation representations, which is remarkably parallel to our results 
in Sections~\ref{sec:2} and~\ref{sec:4}. See 
also \cite{NS12} for a recent survey on equivariant map algebras. 
This connects with our context if 
$\cA = C^\infty(Y,\C)$ and $\Gamma$ acts freely on $Y$, so that 
$\fK := Y \times_\Gamma \fk$ is a Lie algebra bundle over $X := Y/\Gamma$ with 
$\Gamma(\fK) \cong (\fk \otimes \cA)^\Gamma$. 

In \cite{A-T93} an irreducible unitary representation of 
a Lie group $C^\infty_c(X,K)$ is called a {\it non-commutative distribution}. 
In this sense we contribute to the program outlined in \cite{A-T93} 
by classifying the bounded non-commutative distributions for 
$X$ and $K$ compact (at least for the identity component). 
The problem to classify all smooth 
(projective) irreducible unitary representations of gauge groups is still wide 
open, although the classification of their central extensions 
by Janssens and Wockel (\cite{JW10}) is a key step towards this goal. 

We conclude this introduction with a brief discussion of the literature 
on unbounded representation of mapping groups. 
For any, not necessary compact, Lie group $K$, 
the group $C(X,K)$ has unitary representations obtained 
as finite tensor products of evaluation representations. 
However, for some non-compact groups, such as $K = \tilde \SU_{1,n}(\C)$, 
one even has ``continuous'' tensor product representations which 
are irreducible and extend to groups of measurable maps 
(cf.\ \cite{JK85}, \cite{Be79}, \cite{CP87}, \cite{GK03}, \cite{VGG74, GGV80} 
for semisimple target groups, \cite{Gui72} for a general 
discussion and classification results for locally compact target groups, 
\cite{Ar69} for classification results for compact and nilpotent target groups,  
and \cite{Ols82} for an example where the target group 
$\U(\infty)$ is infinite-dimensional). 
In the algebraic context of loop groups, these representations also appear in 
\cite{JK89} which contains a classification of various  types of 
unitary representations generalizing highest weight representations. 
All these representations are most naturally defined 
on groups of measurable maps, so that they neither require a topology nor a smooth 
structure on~$X$.

One of the first references concerning unitary representations 
of groups of smooth maps such as $C^\infty(\R,\SU_2(\C))$ is 
\cite{GG68}, where the authors introduce the concept of a 
{\it derivative representation} which depends only on the 
derivatives up to some order $N$ in some point $t_0 \in \R$. 
Put in our language, derivative representations 
are studied as unitary representations 
of the Lie algebra $\fk_\cA$, where $\cA = \C[[X]]$ is the continuous 
inverse algebra of formal power series. 
In addition to these representations, 
there exist irreducible representations of mapping 
groups defined most naturally on maps of Sobolev $C^1$-maps, the so-called 
energy representations (cf.\ \cite{Is76, Is96}, \cite{AH78}, \cite{A-T93}, \cite{An10}). 
In \cite[\S\S 15, 17]{Is96} one even finds some projective modifications 
of these representations that lift 
to unitary representations of the simply connected covering group. 

For central extensions of loop groups $C^\infty(\bS^1,K)$, 
the most studied class of irreducible representations are the highest weight 
representations from Kac--Moody theory which extend to 
positive energy representations of the semidirect product 
with $\R$, generating the rotation action on $\bS^1$ 
(\cite{PS86}). Other unitary representations are 
the twisted loop modules from \cite{CP86}. These are 
representations of the semidirect product 
$C^\infty(\bS^1, K)\rtimes_\alpha \T$ induced from finite tensor products 
of evaluation representations of $C^\infty(\bS^1, K)$. 
The restrictions of these representations to the loop group are 
bounded but not irreducible. In \cite{To87} 
(see also \cite[\S 5.4]{A-T93}) Torresani studies 
projective unitary ``highest weight representations'' of 
$C^\infty(\T^d,\fk)$, where $\fk$ is compact simple. 
Besides the finite tensor products of evaluation 
representations ({\it elementary representations}) he 
finds finite tensor products of evaluation 
representations of $C^\infty(\T^d,\fk) 
\cong C^\infty(\T^{d-1}, C^\infty(\T,\fk))$, where the representations
 of the target algebra $C^\infty(\T,\fk)$ are projective 
highest weight representations ({\it semi-elementary representations}).

\tableofcontents 

\subsection*{Notation and terminology} 

A {\it (locally convex) Lie group} $G$ is a group equipped with a 
smooth manifold structure modeled on a locally convex space 
for which the group multiplication and the 
inversion are smooth maps. We write $\1 \in G$ for the identity element. 
Its Lie algebra $\g = \L(G)$ is identified with 
the tangent space $T_\1(G)$. The Lie bracket is obtained by identification with the 
Lie algebra of left invariant vector fields. 
A smooth map $\exp_G \: \g \to G$  is called an {\it exponential function} 
if each curve $\gamma_x(t) := \exp_G(tx)$ is a one-parameter group 
with $\gamma_x'(0)= x$. A Lie group $G$ is said to be 
{\it locally exponential} 
if it has an exponential function for which there is an open $0$-neighborhood 
$U$ in $\g$ mapped diffeomorphically by $\exp_G$ onto an 
open subset of $G$. 

If $\pi \: G \to \U(\cH)$ is a norm-continuous (=bounded) 
unitary representation 
of a locally exponential Lie group $G$, then it is automatically 
smooth (\cite[Thms.~III.1.5, IV.1.18]{Ne06}) and its differential 
$\dd\pi \: \g \to \fu(\cH)$ is a continuous representation 
of $\g$ by skew-hermitian bounded operators on $\cH$. We call 
such homomorphisms $\g \to \fu(\cH)$ {\it bounded unitary 
representations} of the Lie algebra~$\g$ because they are representations 
by bounded operators.  
For a {\it $1$-connected}, i.e., connected and simply connected, Lie group $G$,  
the bounded unitary representations 
of the Lie algebra $\g$ are in one-to-one correspondence with the 
unitary representations $\pi \: G \to \U(\cH)$ which are smooth maps 
with respect to the Lie group structure on $\U(\cH)$ defined by the norm. 
All groups considered in the present paper are locally exponential.

\section{Multiplicative characters on continuous inverse algebras} 
\mlabel{sec:1} 

In this section we provide a key ingredient of our classification 
results. We show that every function $\phi \: \cA \to \C$ on 
a continuous inverse algebra $\cA$ which is multiplicative and polynomial 
of degree $N$ is a product of $N$ algebra homomorphisms 
$\chi_j \: \cA \to \C$ (Theorem~\ref{thm:2.5}). This will be crucial 
in Section~\ref{sec:3} to deal with the Lie algebra $\su_2(\cA)$. 
For unital commutative Banach algebras, this result is already 
known (\cite{NS11}), so the new point is its extension 
to locally convex algebras. 

\subsection{Tensor products of continuous inverse algebras}

\begin{defn} (a) 
A unital locally convex associative algebra (with continuous multiplication) 
$\cA$ over $\bk \in \{ \R,\C\}$ 
is called a {\it continuous inverse algebra} 
(cia for short) if its unit group $\cA^\times$ is open and the inversion map 
$\cA^\times \to \cA, a \mapsto a^{-1}$ is continuous. 
We write $\Gamma_\cA := \Hom(\cA,\C) \setminus \{0\}$ for the set of 
non-zero algebra homomorphisms; the {\it characters of $\cA$}. 

If $\cA$ has no unit, then we say that $\cA$ is a cia if 
the unital algebra $\cA_+ := \bk \oplus \cA$ with the product 
\[ (t,a) (s,b) := (ts, t b + sa + ab) \] 
is a cia. 

(b) An {\it involutive cia} is a complex cia $\cA$, endowed with an antilinear 
antiisomorphism $a \mapsto a^*$. This map is called the {\it involution of $\cA$}. 
Then $\cA_\R := \{ a \in \cA \: a^* = a\}$ is a real subspace of $\cA$ 
which is a real cia if $\cA$ is commutative. 

A character $\chi \: \cA \to \C$ is said to be {\it involutive} if 
$\chi(a^*) = \oline{\chi(a)}$ holds for $a \in \cA$. 
We write $\Gamma_\cA^*$ for the set of non-zero involutive characters of~$\cA$. 
Note that, for any character $\chi$, we obtain another character 
$\chi^*(a) := \oline{\chi(a^*)}$, and that $\chi$ is involutive if and only if 
$\chi^* = \chi$. 
\end{defn}

\begin{ex} \mlabel{ex:cia} 
(a) For any compact smooth manifold $X$ and 
$k \in \N_0 \cup \{\infty\}$, the algebra 
$\cA := C^k(X,\bk)$ is a unital cia. Its characters are of the form 
$\chi(f) = f(x)$ for some $x \in X$ (cf.\ \cite[Thm.~4.3.1(b)]{Wag11b}). 

(b) If $X$ is non-compact, then 
$C^k(X,\bk)$ is no longer a cia because spectra of elements of a cia are compact. 
In this case the algebra $\cA := C^k_c(X,\bk)$ of compactly supported $C^k$-functions 
is a cia (cf.\ \cite[Prop. 7.1]{Gl02}). 
Its characters are of the form $\chi(f) =f(x)$ 
for some $x \in X$ (after extending the character to the unital algebra 
$\cA_+$, the arguments in the proof of \cite[Thm.~4.3.1(b)]{Wag11b} 
carry over). 

(c) For $d \in \N$, let 
$\cA := \bk[[X_1,\ldots, X_d]]$ denote the algebra of formal 
power series in the commuting variables $X_1,\ldots, X_d$. We endow 
$\cA$ with the Fr\'echet topology defined by the seminorms 
\[ p_\alpha(a) := |a_\alpha|, \quad \mbox{ where } \quad 
a = \sum_{\alpha \in \N_0^d} a_\alpha X^\alpha. \] 
Then $\cA$ is a unital cia with unit group 
$\cA^\times = \{ a \: a_0 \not= 0\}$ and the unique maximal ideal 
$\fm = (X_1, \ldots, X_d)$ is a hyperplane. 
Accordingly $\Gamma_\cA = \{ \eps\}$ with 
$\eps(a) = a_0$. 

(d) Let $K \subset \C^n$ be a compact set. For each open neighborhood 
$U$ of $K$, we denote by $\cO^{\infty}(U)$ the Banach algebra of 
bounded holomorphic functions on $U$, equipped with the supremum norm. 
Let $\cO_{\mathrm{an}}(K)$ be the algebra of germs of holomorphic functions around $K$.
Then $\cO_{\mathrm{an}}(K)$ is the direct limit of the algebras $\cO^{\infty}(U)$, where $U$ runs over the 
open neighborhoods of $K$, and we equip $\cO_{\mathrm{an}}(K)$ with the direct limit topology.
This makes $\cO_{\mathrm{an}}(K)$ into a complete unital cia (cf.~\cite{Wae54}).
If $K$ is a rationally convex subset of $\C^n$ (meaning 
that, for every $x \in \C^n \setminus K$, 
there is a polynomial $p$ such that $p(x) \notin p(K)$), 
then  
every character is of the form 
$\chi (f) = f(x)$ for some $x\in K$,
and $\Gamma_{\cO_{\mathrm{an}}(K)}$ is homeomorphic to $K$ (cf.~\cite{Bi07}).

(e) The algebra $\cA := \cS(\R^d,\R)$ of Schwartz functions on 
$\R^d$ is a cia with $\Gamma_\cA \cong \R^d$ 
(all characters are given by evaluations in 
points of $\R^d$). This follows from the automatic 
continuity of characters on cias, the density of 
$C^\infty_c(\R^d,\R)$, and (b).

(f) If $\sigma \: \Gamma \times X \to X$ is an action of the group 
$\Gamma$ by diffeomorphisms on the compact manifold $X$, then the subalgebra 
$C^\infty(X,\bk)^\Gamma$ of $\Gamma$-invariant functions is also a cia. 
More generally, if $\cA$ is a unital cia and $\Gamma \subeq \Aut(\cA)$, 
then the subalgebra $\cA^\Gamma$ of $\Gamma$-fixed points is a cia. 
Here the main point is to observe that $\cA^\Gamma \cap \cA^\times \subeq 
(\cA^\Gamma)^\times$. 
\end{ex}

\begin{defn} \mlabel{def:3.1} (a) If $E$ and $F$ are locally convex spaces, then the 
{\it projective topology} on $E \otimes F$ is defined by the seminorms 
\[  (p \otimes q)(x) = \inf \Big\{ \sum_{j=1}^n p(y_j) q(z_j) \: x = 
\sum_{j=1}^n  y_j \otimes z_j\Big\}, \]
where $p$ and $q$ are continuous seminorms on $E$ and $F$, respectively 
(cf.~\cite[Prop.~43.4]{Tr67}). 
It has the universal property that 
continuous bilinear maps $E \times F \to G$, $G$ a locally convex space, 
are in one-to-one correspondence with continuous linear maps 
$E \otimes F \to G$. 
One likewise defines the projective 
topology for tensor products of finitely many factors.

(b) Now let $\cA^{\hat\otimes N}$ denote the 
completed projective $N$-fold 
tensor product of a commutative cia $\cA$. 
It has the universal property that continuous linear maps 
$\cA^{\hat\otimes N} \to E$ 
to a complete locally convex space $E$ are in one-to-one correspondence 
with continuous $N$-linear maps $\cA^N \to E$. 
From the universal property and the associativity 
of projective tensor products it  
easily follows that $\cA^{\hat\otimes N}$ carries a natural 
unital commutative algebra structure, determined by 
\begin{eqnarray*}
(a_1 \otimes \cdots \otimes a_N)(b_1 \otimes \cdots \otimes b_N)
:=a_1b_1 \otimes \cdots \otimes a_Nb_N \quad \mbox{ for } \quad 
a_i, b_i \in \cA. 
\end{eqnarray*} 
The symmetric group $S_N$ acts on $\cA^{\hat\otimes N}$ 
by permuting the tensor factors, i.e.
\begin{eqnarray*}
\sigma(a_1 \otimes \cdots \otimes a_N):=a_{\sigma^{-1}(1)} \otimes \cdots 
\otimes a_{\sigma^{-1}(N)}, \quad \sigma \in S_N, a_j \in \cA.
\end{eqnarray*}
The fixed point algebra 
$\oline S^N(\cA) := (\cA^{\hat\otimes N})^{S_N}$ 
is also a unital commutative algebra which is a closed subalgebra of 
$\cA^{\hat\otimes N}$, hence complete. It 
is topologically generated by tensors of the form  
\begin{eqnarray*}
a_1 \vee \cdots \vee a_N:=\frac{1}{N!}
\sum_{\sigma \in S_N}a_{\sigma(1)} \otimes \cdots \otimes a_{\sigma(N)},
\end{eqnarray*}
and by polarization it is actually generated by the diagonal elements
\[ a^{\vee N} = a \vee \cdots \vee a=a \otimes \cdots \otimes a, \quad a \in \cA.\] 
\end{defn}

We do not know if the (completed) tensor product of two cias is always 
a cia. The following theorem shows that this is true for commutative ones. 

\begin{thm} If $\cA_1$ and $\cA_2$ are commutative unital continuous inverse algebras, 
then their completed projective 
tensor product $\cA_1 \hat\otimes \cA_2$ is also a continuous inverse algebra. 
\end{thm}

\begin{prf} Let $\cB := \cA_1 \otimes \cA_2$ denote the projective tensor product of 
$\cA_1$ and $\cA_2$. Then $\cB$ is a unital locally convex commutative algebra. 
The projective topology on $\cB$ is defined by the seminorms $p \otimes q$, 
where $p$ and $q$ are continuous seminorms on $\cA_1$ and $\cA_2$, respectively 
(cf.\ Definition~\ref{def:3.1}). 
Suppose that $p$ and $q$ are submultiplicative. For 
$x = \sum_j y_j \otimes z_j$ and 
$x' = \sum_j y_j' \otimes z_j'$ we have  
$xx' = \sum_{j,k} y_j y_k' \otimes z_j z_k'$. From 
\[ \sum_{j,k} p(y_j y_k') q(z_j z_k') 
\leq \sum_{j,k} p(y_j) q(z_j)  p(y_k') q(z_k')
= \sum_j  p(y_j) q(z_j)  \sum_k  p(y_k') q(z_k') \] 
we then derive that 
\[ (p \otimes q)(xx') \leq (p \otimes q)(x)(p \otimes q)(x'),\]
i.e., that $p \otimes q$ is submultiplicative. 

According to Turpin's Theorem (\cite{Tu70}), commutative continuous inverse algebras 
have a defining family of submultiplicative seminorms, i.e., they are projective limits 
of commutative Banach algebras. As we have seen above, this property is inherited 
by $\cB$. In particular, $\cB$ embeds into a topological product of Banach algebras, 
which implies that the inversion map $\cB^\times \to \cB$ is continuous. It therefore 
remains to show that $\cB^\times$ is a neighborhood of $\1$. 

It is clear that 
$\Gamma_\cB  := \Hom(\cB,\C)$ can be identified with the product set 
$\Gamma_{\cA_1} \times \Gamma_{\cA_2}$. As subsets of the dual space $\cA_j'$, the set  
$\Gamma_{\cA_j}$ is closed and equicontinuous.
Indeed, let $\1 + U$ be an open neighborhood of $\1$ in $\cA^\times$, with $U$
a circular neighborhood of 0. If $u \in U$, then $\1 + \gamma u$ is invertible
for all $\gamma \in \bk$ with $|\gamma|\leq 1$, and thus  
$\chi(\1 + \gamma u) \neq 0$ for all $\chi \in \Gamma_{\cA}$. This implies 
that $|\chi(u)| < 1$ for all $\chi$, which shows that $\Gamma_{\cA}$ is equicontinuous, 
hence weak-$*$-compact (cf.~\cite[Prop.~32.8]{Tr67}). In particular, the Gelfand maps 
\[ \cG_j \: \cA_j \to C(\Gamma_ {\cA_j}), \quad \cG_j(a)(\chi) := \hat a(\chi) := \chi(a) \] 
are continuous homomorphisms satisfying $\|\cG_j(u)\| < 1$ for $u \in U$.  
We thus obtain on $\Gamma_\cB$ the structure of a compact Hausdorff space, 
and the Gelfand map
\[ \cG \: \cB \to C(\Gamma_\cB), \quad a_1 \otimes a_2 \mapsto \hat{a_1} \otimes \hat{a_2}\] 
for $\cB$ is continuous by the universal property of the projective tensor product.

We claim that, for $\|\cG(b)\| < 1$, the Neumann series 
$\sum_{n = 0}^\infty b^n$ converges in the completion $\hat\cB$ of $\cB$. 
Let $r \: \cB \to \R$ be a submultiplicative seminorm and write $q_r \: \cB \to \cB_r$ for the 
map into the corresponding Banach algebra $\cB_r$, which is the completion of $\cB/r^{-1}(0)$ with 
respect to the norm induced by~$r$. Then $\Gamma_{\cB_r}$ is a subset of $\Gamma_{\cB}$, so that 
$\|\cG(b)\| < 1$ implies that $\Spec(q_r(b))$ 
(which is the image of $\cG(b)$ restricted to $\Gamma_{\cB_{r}}$)
is contained in the open unit disc and therefore 
$\sum_{n = 0}^\infty 
q_r(b)^n$ converges in $\cB_r$. Since $r$ was arbitrary, it follows that 
$\sum_{n = 0}^\infty b^n$ converges in $\hat\cB$, which can be identified with a subset of 
the product space $\prod_r \cB_r$. 

We conclude that $\hat\cB$ is a commutative unital algebra with an open unit group 
and continuous inversion. This completes the proof. 
\end{prf}

\begin{cor} \mlabel{cor:2.3} 
For any commutative unital continuous inverse algebra $\cA$, 
the completed projective tensor powers 
$\cA^{\hat\otimes N}$  are continuous inverse algebras. 
\end{cor}

\subsection{Multiplicative characters}

We classify the holomorphic multiplicative characters of $\cA$. For this,
we shall need the following proposition, which is a generalization of 
the corresponding assertion concerning Banach \mbox{algebras (\cite{NS11})}:

\begin{prop}\mlabel{prop:2.4} 
If $\cA$ is a complex commutative continuous inverse algebra and 
$\Gamma  \subeq \Aut(\cA)$ a finite subgroup, 
then each character $\chi \: \cA^\Gamma \to \C$ 
extends to a character of $\cA$. 
\end{prop}

\begin{prf} Let $\cI \subeq \cA^\Gamma$ denote the kernel of $\chi$. 
Then $\cI$ is a proper ideal of $\cA^\Gamma$, 
and \cite[Lemma~3.1]{NS11} implies that 
$\cA \cI$ is a proper ideal of $\cA$, hence contained 
in a maximal ideal $\cM$. Now $\cA/\cM$ is a complex division algebra and a cia 
(\cite[Lemma~B.9]{Wag11}), so that \cite{Are47} 
implies that $\cA/\cM \cong \C$. Therefore the quotient map 
$\cA \to \cA/\cM$ is a character extending $\chi$. 
\end{prf}

\begin{rem} \mlabel{rem:bd} 
(a) Unfortunately, the analog of Proposition~\ref{prop:2.4} 
for involutive characters is false. Here is the minimal example. 
Let $\bD := \C^2$ denote the $2$-dimensional involutive algebra, 
where the algebra structure is given by pointwise
multiplication and $(z_1, z_2)^* := (\oline{z_2}, \oline{z_1})$. 
Then $\Gamma_\bD = \{ \chi_1, \chi_2\}$ is a $2$-element set with 
$\chi_j(z) = z_j$ and $\chi_1^*= \chi_2$. In particular, the involutive 
algebra $\bD$ has no involutive characters. However, 
$\Gamma := \Aut(\bD,*)$ is the two-element group $\{\id,\tau\}$ with 
$\tau(z_1, z_2) = (z_2, z_1)$ and 
$\bD^\Gamma = \Delta_\C$ is the diagonal subalgebra, 
on which $\chi(z,z) := z$ defines an involutive character. 

(b) An involutive cia is called {\it hermitian} if, 
for each hermitian element $a = a^*$, the spectrum 
$\sigma_\cA(a)$ is contained in $\R$, i.e., $a + z \1$ is invertible 
for $z \not\in \R$. For a hermitian cia any character 
$\chi \: \cA \to \C$ is involutive because $\chi(a) \in \R$ 
for $a = a^*$. 

(c) The algebra $\bD$  is not hermitian because the element 
$a = (i,-i)$ is hermitian with $\sigma_\bD(a) = \{ \pm i\}$. 
Note that $\chi(z) := z_1 z_2$ is a multiplicative 
character $(\bD,\cdot) \to (\C,\cdot)$ 
satisfying $\chi(a^*) = \oline{\chi(a)}$ for 
every $a \in \bD$. 
\end{rem}

\begin{theorem}\mlabel{thm:2.5} 
Let $\cA$ be a commutative unital continuous inverse algebra. 
Then, for every continuous polynomial multiplicative map $\phi: \cA \rightarrow (\C,\cdot)$ 
of degree $N$, 
there exist finitely many continuous algebra homomorphisms 
$\chi_1, \ldots, \chi_N \: \cA \to \C$ such that
$\varphi=\chi_1 \cdots \chi_N.$ 
\end{theorem}

\begin{proof} Write 
$\phi(a) = \psi(a \vee \cdots \vee a)$ for a continuous linear map 
$\psi \: S^N(\cA) \to \C.$ 
For the diagonal generators of $S^N(\cA)$ 
(cf.\ Definition~\ref{def:3.1}), we then have 
\[ \psi( a^{\vee N} b^{\vee N}) 
= \psi( (ab)^{\vee N}) 
= \phi(ab) = \phi(a) \phi(b) 
=  \psi( a^{\vee N}) \psi(b^{\vee N}).\] 
From the linearity of $\psi$ and its multiplicativity 
on a set of topological  linear generators, it now follows that $\psi$ is an algebra homomorphism. 

The continuity of $\psi$ further implies that it 
extends to a continuous linear map on the completion 
$\oline S^N(\cA)$, which is the fixed point algebra for the canonical 
$S_N$-action on the completion of 
$\cA^{\otimes N}$.
From Corollary~\ref{cor:2.3} we know that this completion is a continuous 
inverse algebra, so that Proposition~\ref{prop:2.4} shows that 
$\psi$ extends to a continuous character 
$\chi: \cA^{\otimes N} \rightarrow \C$.
Then
\[ \phi(a)=\chi(a \otimes \cdots \otimes a)
=\prod_{i = 1}^n \chi(\1^{\otimes i-1} \otimes a \otimes \1^{\otimes n-i}) 
=\chi_1(a)\cdots \chi_N(a),\] 
where $\chi_i \: \cA \rightarrow \C$ is the character
$\chi_i(a):=\chi(\1^{\otimes i-1} \otimes a \otimes \1^{\otimes n-i})$.  
Since every character of $\cA$ is automatically continuous (because $\cA^\times$ is open), 
this proves the theorem.
\end{proof}

\begin{cor} \mlabel{cor:4.9} 
If $\cA$ is a complex unital cia and 
$\phi \: (\cA,\cdot) \to \C$ holomorphic and multiplicative, then 
there exist finitely many continuous algebra homomorphisms 
$\chi_j \: \cA \to \C$ such that $\varphi=\chi_1 \cdots \chi_N.$ 
\end{cor}

\begin{prf} Since $\phi$ is holomorphic, its restriction to the subalgebra $\C \1$ 
is holomorphic and multiplicative, hence of the form 
$\phi(z\1) = z^n$ for some $n \in \N_0$. This implies that 
$\phi(za) = z^n \phi(a)$ for $z\in \C, a \in \cA$.
Now $\phi$ has a Taylor expansion in homogeneous
polynomials (cf.~\cite{Ne06}), and the only nonzero term is
the one of degree $n$. Hence the preceding theorem applies.
\end{prf}

\begin{ex} If $X$ is a compact manifold, 
then $\cA = C^\infty(X)$ is a unital continuous 
inverse algebra with spectrum $\Gamma_\cA = X$ (Example~\ref{ex:cia}(b)). 
As the completed tensor powers are given by $\cA^{\hat\otimes n} 
\cong C^\infty(X^n)$ 
(\cite[Thm.~4]{Ma02}) and the $S_n$-action on this algebra is induced by the natural 
action of $S_n$ on $X^n$, the algebra $S^n(\cA)$ consists of the 
smooth symmetric functions on $X^n$. In particular, its spectrum is the quotient 
$X^n/S_n$, which can be identified with the set of $n$-element multisubsets of $X$. 
As the simple example $X = \R$ already shows, this space is not a smooth manifold. 
\end{ex}

\section{Irreducible \texorpdfstring{$*$}{*}-representations of 
\texorpdfstring{$\g \otimes \cA$}{g(A)}} 
\mlabel{sec:2}

Let $\fk$ be a compact semisimple Lie algebra and 
$\g := \fk_\C$ its complexification. In this section we develop an 
analog of the classical Cartan--Weyl theory of highest weight representations 
for bounded irreducible unitary representations of $\fk_\cA := \fk \otimes_\R \cA_\R$.
Here and in the remainder of this section, $\cA$ is an involutive 
unital commutative continuous inverse algebra.
Our main 
result is a classification of the irreducible bounded unitary representations of 
$\fk_\cA$ as finite tensor products of evaluation representations. 

\subsection{Triangular decomposition}

We write $x \mapsto \oline x$ for the complex conjugation on $\g=\fk_\C$ 
and put $x^* := - \oline x$. We then have 
$\fk = \{ x \in \g\: x^* = -x \}.$
If $\ft \subeq \fk$ is maximal abelian, then 
$\fh := \ft_\C$ is a Cartan subalgebra. Let $\Delta \subeq \fh^*$ be 
the corresponding root system, so that we have the root decomposition 
\begin{eqnarray*}
\g = \mathfrak{h} \oplus \bigoplus_{\alpha \in \Delta} \g_{\alpha}.
\end{eqnarray*}
Note that $\oline{\fh} = \fh$ and $\oline{\g_\alpha} = \g_{-\alpha}$. 
We write $\check \alpha \in \fh$ for the coroot associated to 
$\alpha \in \Delta$, i.e., the unique element $\check \alpha
 \in [\g_\alpha, \g_{-\alpha}]$ with $\alpha(\check \alpha) = 2$. 
Fix a positive system $\Delta^+$, and let
$\Pi = \{ \alpha_1,\ldots, \alpha_r\}$ denote the corresponding 
simple roots. 

Then 
$\g(\mathcal{A}):=\g \otimes_\C \mathcal{A},$ 
equipped with Lie bracket 
\begin{eqnarray*}
\left[x_1 \otimes a_1, x_2 \otimes a_2\right]:=[x_1,x_2] \otimes a_1a_2
\end{eqnarray*}
is a locally convex Lie algebra with respect to the natural tensor product 
topology, for which $\g(\cA) \cong \cA^{\dim \g}$ as a topological vector space. 
The antilinear antiautomorphisms $*$ of $\g$ and $\cA$ combine to the 
antilinear antiautomorphism of $\g(\cA)$, defined by 
\[ (x \otimes a)^* := x^* \otimes a^*.\] 
The Lie algebra $\fk_{\cA} = \fk \otimes_\R \cA_\R$ 
is the corresponding real form;
$ \fk_\cA = \{ z \in \g(\cA) \: z^* = - z \}$. 
We define\footnote{\label{note1} 
In view of the difference in sign conventions \label{footnotepage}
for holomorphic induction and 
highest weight representations, 
we define
$\fg^+$ to be the span of the root spaces corresponding to \emph{negative} roots.} 
\[ \fg^\pm := \sum_{\alpha \in \mp \Delta^+} (\g_\alpha \otimes \cA) 
\quad \mbox{ and } \quad 
\g^0 := \fh \otimes \cA \] 
to obtain the triangular decomposition 
\[ \g(\cA) = \fg^+ \oplus \g^0 \oplus \fg^-.\] 

\subsection{Inducible functionals}
Bounded $*$-representations correspond
to so-called inducible functionals. 
In this section, we will classify the inducible 
functionals of $\fg(\cA)$.

\begin{defn} \mlabel{def:bsrep}

(a) For a real topological Lie algebra $\fu$, a {\it bounded unitary representation} 
is a pair $(\pi, \cH)$, where $\cH$ is a complex Hilbert space and 
$\pi \: \fu \to \fu(\cH)$ a continuous homomorphism of Lie algebras. 

(b) If $(\g, *)$ is a complex topological Lie algebra and 
$x \mapsto x^*$ a continuous antilinear involutive antiisomorphism, then 
a {\it bounded $*$-representation of $\g$} is a pair $(\pi, \cH)$, 
where $\cH$ is a complex Hilbert space and 
$\pi \: \g \to B(\cH) = \gl(\cH)$ a continuous homomorphism of Lie algebras with 
$\pi(x^*) = \pi(x)^*$ for $x \in \g$. 
Then the restriction to the real form 
$\fu := \{ x \in\g \: x^* = - x\}$ is a bounded unitary representation. 
Conversely, the complex linear extension of every bounded 
unitary representation 
of $\fu$ to $\g$ is a bounded $*$-representation. 
\end{defn}

\begin{prop} \mlabel{prop:6.1} Let $(\pi, \cH)$ be a bounded $*$-representation 
of $\g(\cA)$. Then the following assertions hold: 
\begin{description}
\item[\rm(i)] $\cE := \cH^{\fg^-} = (\pi(\fg^+)\cH)^\bot$ carries a
$*$-representation $\rho$ of the commutative subalgebra $\g^0 = \fh \otimes \cA$. 
\item[\rm(ii)] $\cE$ generates the $\g(\cA)$-module $\cH$. 
\item[\rm(iii)] There exists an $N \in \N$ with $\pi(\fg^+)^N 
= \pi(\fg^-)^N = \{0\}$. 
\item[\rm(iv)] The restriction map 
$R \: \pi(\g(\cA))' \to \rho(\g^0)', B \mapsto B\res_{\cE}$ is an isomorphism 
of von Neumann algebras. 
\item[\rm(v)] $(\pi, \cH)$ is irreducible if and only if 
$\dim \cE = 1$. 
\end{description}
\end{prop}

\begin{prf} (i) The relation $(\fg^+)^* = \fg^-$ implies that 
$\pi(\fg^+)^* = \pi(\fg^-)$, which leads to 
\[ \cH^{\fg^-} = \ker(\pi(\fg^-)) = (\pi(\fg^+)\cH)^\bot.\] 
Further, $[\g^0, \g^-] \subeq \g^-$ implies that 
this closed subspace is invariant under $\g^0$. 

(ii) Since, for every $\g(\cA)$-invariant subspace 
$\cK \subeq \cH$, the orthogonal complement $\cK^\bot$ is also invariant, 
it suffices to show that any non-zero invariant subspace 
$\cK$ intersects $\cE$ non-trivially. This amounts to showing that, 
if $\cH$ is non-zero, then $\cE$ is non-zero. 

To this end, we integrate the representation of the 
finite-dimensional involutive Lie algebra 
$\pi \: \g \to \gl(\cH)$ to a holomorphic representation 
of the corresponding $1$-connected group $\pi_G \: G \to \GL(\cH)$. 
Let $\ft := \fk \cap \fh$. Then $T := \exp_G(\ft)$ is a torus, 
so that the boundedness of $\pi$ implies that $\cH$ 
decomposes into  finitely many $\fh$-weight spaces 
\[ \cH = \bigoplus_{\beta \in \fh^*} \cH_\beta.\] 
From the relation $\g_\alpha(\cA) \cH_\beta \subeq \cH_{\beta + \alpha}$ 
and the finiteness of the decomposition of $\cH$, we derive the existence 
of a minimal $N \in \N$ with $\pi(\fg^-)^N = \{0\}$.  
Then $\pi(\fg^-)^{N-1}\cH$ is non-zero and contained in~$\cE$. 

(iii) We have already seen that 
$\pi(\fg^-)^N = \{0\}$, so that (ii) follows from 
$\pi(\fg^+) = \pi(\fg^-)^*$. 

(iv) Since $\pi(\g(\cA))'$ commutes in particular with 
$\pi(\fg^-)$, it leaves the subspace $\cE$ invariant, so that 
$R$ is well-defined. Since $\cE$ generates the $\g(\cA)$-module $\cH$, 
the map $R$ is injective. To see that it is surjective, 
it suffices to show that its range contains all projections 
of $\rho(\g^0)'$. Here we use that 
each von Neumann algebra is generated by hermitian  projections 
(\cite[Chap.~1, \S 1.2]{Dix96}) 
and that images of von Neumann algebras under restriction maps 
are von Neumann algebras (\cite[Chap.~1, \S 2.1, Prop.~1]{Dix96}). 
So let $P = P^* =P^2 \in \rho(\g^0)'$ be a hermitian 
projection and $\cE_0 := P(\cE)$ be its range. 

If we denote by $U(\fg(\cA))$ the universal enveloping algebra, then
 $\cH_0 := \oline{U(\g(\cA))\cE_0}$ is the closed $\g(\cA)$-invariant 
subspace generated by $\cE_0$. From 
$U(\g(\cA)) = U(\fg^+)U(\g^0)U(\fg^-)$ we derive that 
\[ \cH_0 \subeq \oline{U(\fg^+)\cE_0} 
\subeq \cE_0 + \pi(\fg^+)\cH
\subeq \cE_0 + \cE^\bot,\] 
which implies that $\cH_0 \cap \cE = \cE_0$. 
We conclude that the orthogonal projection 
$\tilde P \: \cH \to \cH_0$, which is contained in $\pi(\g(\cA))'$, 
satisfies $\tilde P\res_{\cE} = P$. Therefore $R$ is surjective. 

(v) In view of Schur's Lemma, $(\pi, \cH)$ is irreducible if and only if 
$\pi(\g(\cA))' = \C \1$. According to (iv), this is equivalent to 
the irreducibility of $(\rho, \cE)$. As $\g^0$ is abelian, this is equivalent 
to $\dim \cE = 1$. 
\end{prf}

\begin{rem} (Disintegration of bounded representations) 
For every bounded $*$-representation $(\pi, \cH)$ of $\g(\cA)$, 
every irreducible representation of the $C^*$-algebra $\cA := C^*(\pi(\g(\cA)))$ 
generated by $\pi(\g(\cA))$   
defines an irreducible bounded $*$-representations of 
$\g(\cA)$. This implies that $\pi$ is a direct 
integral of irreducible ones (cf.\ \cite{Dix77} in the separable case and 
\cite{He82} for inseparable representations). We therefore understand 
the structure of bounded $*$-representations if we know the irreducible 
representations. 
\end{rem}

\begin{defn} We call an involutive  linear functional 
$\lambda \: \g^0 \to \C$ {\it inducible} 
if there exists a bounded $*$-representation 
$(\pi, \cH)$ of $\g(\cA)$ with 
$(\rho, \cE) \cong (\lambda,\C)$, i.e., 
if $\lambda$ occurs as the $\g^0$-weight on $\cE$ for some 
irreducible bounded representation of~$\g(\cA)$.
\end{defn} 

\begin{lem}\label{findimeq} If $(\pi_1, \cH_1)$ and $(\pi_2, \cH_2)$ are 
irreducible representations with $\g^0$-weights 
$\lambda_1$ and $\lambda_2$ on $\cE$, 
then $\pi_1 \cong \pi_2$ if and only if $\lambda_1 = \lambda_2$. 
\end{lem}

\begin{prf} Suppose that $\lambda_1 = \lambda_2 = \lambda$. 
Consider the direct sum representation 
$\pi := \pi_1 \oplus \pi_2$ on 
${\cH := \cH_1 \oplus \cH_2}$. 
Then $\cE = \cE_1 \oplus \cE_2$ is $2$-dimensional with 
$\rho(x)(v,w) = (\lambda(x)v, \lambda(x)w)$. 
In particular, we obtain $\rho(\g^0)' \cong M_2(\C)$ 
for the commutant. We conclude that $\pi(\g(\cA))' \cong M_2(\C)$ 
(Proposition~\ref{prop:6.1}(iv)). 
If the representations $\pi_1$ and $\pi_2$ were not equivalent, 
then we would have obtained $\pi(\g(\cA))' \cong \C^2$ by Schur's Lemma. 
This 
shows that $\pi_1 \cong \pi_2$. 
Conversely, suppose that $\pi_1 \simeq \pi_2$. Then the 
intertwiner $U \: \cH_1 \rightarrow \cH_2$ 
restricts to an intertwiner
$U|_{\cE_1} \: \cE_1 \rightarrow \cE_2$
of the one-dimensional representations $\rho_1$
and $\rho_2$, so \mbox{that $\lambda_1 = \lambda_2$}. 
\end{prf}

\begin{defn} In view of the preceding lemma, we write 
$(\pi_\lambda, \cH_\lambda)$ for the unique irreducible 
representation with $(\rho, \cE) \cong (\lambda,\C)$. 
We call $\lambda$ the {\it highest weight of $\pi_\lambda$}. 
\end{defn}

\begin{lem} \mlabel{lem:sum} 
If $\lambda$ and $\mu$ are inducible, then so is their sum 
$\lambda + \mu$. 
\end{lem}

\begin{prf} We consider the $*$-representation $(\pi, \cH)$ with 
\[ \cH := \cH_\lambda \otimes \cH_\mu \quad \mbox{ and } \quad 
\pi := \pi_\lambda \otimes \1 + \1 \otimes \pi_\mu.\] 
Then $\cF := \cE_\lambda \otimes \cE_\mu \subeq \cE$ is a one-dimensional 
subspace on which $\g^0$ acts by the weight $\lambda + \mu$. 
Since $\cH_0 := \oline{U(\g(\cA)) \cF} \subeq \cH$ is a 
$\g(\cA)$-submodule with $\cF = (\cH_0)^{\fg^-}$ 
(Proposition~\ref{prop:6.1}(iv)), it carries an irreducible representation with 
highest weight $\lambda + \mu$. 
\end{prf}

\begin{defn} Let $\cA$ be a commutative involutive cia and 
$\chi \: \cA \to \C$ an involutive character. 
Then $\ev_\chi := \id \otimes \chi \: \g(\cA) \to \g(\C) \cong \g$ 
is an involutive algebra homomorphism. 
If $(\rho,\cH)$ is a $*$-representation of $\g$ 
with highest weight $\lambda$, then the 
representation $\pi_{\chi,\rho} := \rho \circ \ev_\chi$ 
of $\g(\cA)$ on $\cH$ is called an 
irreducible {\it evaluation representation}. 
This is an irreducible $*$-representation with highest 
weight~$\lambda \otimes \chi$. 
\end{defn}

The proof of the following theorem builds 
on the main result of Section~\ref{sec:3} (Theorem~\ref{thm:6.x}) 
which deals with the special case $\g = \fsl_2(\C)$. 

\begin{thm} \mlabel{thm:ind-charac} 
All functionals of the form 
$\lambda \otimes \chi$, where $\chi \in \Gamma_\cA^*$ is an involutive character 
and $\lambda \in \fh^*$ is dominant integral, i.e., 
$\lambda(\check \alpha) \in \N_0$ for $\alpha \in \Pi$, are 
inducible. Conversely, any inducible functional is such a finite sum. 
\end{thm}

\begin{prf} The definition of the evaluation representations 
shows that any functional of the form 
$\lambda \otimes \chi$, $\chi \in \Gamma_\cA^*$  and $\lambda \in \fh^*$ 
dominant integral, is inducible. Further, Lemma~\ref{lem:sum} 
implies that any sum of such 
functionals is inducible. 

We now show that any inducible functional is of this form. 
Suppose that $\lambda$ is inducible 
and that $(\pi, \cH) := (\pi_\lambda, \cH_\lambda)$ is the corresponding representation 
of $\g(\cA)$. Let $\alpha \in \Delta$ and 
\[ \g^\alpha(\cA) := \g_\alpha(\cA) + \g_{-\alpha}(\cA) 
+ \check \alpha \otimes \cA \cong \fsl_2(\cA).\] 
Then $\cE$ is annihilated by $\g_\alpha(\cA)$ and 
generates a $\g^\alpha(\cA)$-subrepresentation 
$(\pi_\alpha, \cH_\alpha)$. Then 
\[ U(\g^\alpha(\cA)) \cE = U(\g_{-\alpha}(\cA)) \cE 
\subeq \cE + \pi_\alpha(\g_{-\alpha}(\cA)) \cH_\alpha \] 
implies that 
\[ \cE_\alpha := \cH_\alpha^{\g_\alpha(\cA)} = (\pi_\alpha(\g_{-\alpha}(\cA))\cH_\alpha)^\bot 
= \cE \] 
is one-dimensional, so that $(\pi_\alpha, \cH_\alpha)$ is irreducible 
with $\cE_\alpha = \cE$. Therefore 
$\lambda\res_{\check \alpha \otimes \cA}$ is  inducible. 
The main result of Section~\ref{sec:3} (Theorem~\ref{thm:6.x}) 
asserts that there exist finitely 
many pairwise different involutive characters $\chi_1, \ldots, \chi_N \in \Gamma_\cA$ 
and $m_j \in \N_0$ with 
\[ \lambda(\check\alpha \otimes a) 
= \sum_{j = 1}^N m_j\chi_j(a) \quad \mbox{ for } \quad a \in \cA.\] 

Recall that $\Pi = \{ \alpha_1, \ldots, \alpha_r\}$ is the set of 
simple roots. Then 
$\check \alpha_1, \ldots, \check \alpha_r$ is a basis of $\fh$, 
and we obtain finitely many pairwise different 
involutive characters $\chi_j \in \Gamma_\cA$ and 
$m_{ij} \in \N_0$ with 
\[ \lambda(\check \alpha_i \otimes a) = 
\sum_j m_{ij} \chi_j(a) \quad \mbox{ for } \quad a \in \cA.\] 
Define $\lambda_j \in \fh^*$ by 
$\lambda_j(\check \alpha_i) = m_{ij}$ and note that 
$\lambda_j$ is dominant integral. We now have for each~$i$ 
\[ \lambda(\check \alpha_i \otimes a) 
= \sum_j (\lambda_j \otimes \chi_j)(\check \alpha_i \otimes a) 
 \quad \mbox{ for } \quad a \in \cA, \] 
so that 
$\lambda = \sum_j \lambda_j \otimes \chi_j.$
\end{prf} 

\subsection{Bounded \texorpdfstring{$*$}{*}-representations}
The classification of inducible functionals now 
yields the irreducible $*$-representations \mbox{of $\fg(\cA)$}.

\begin{lem}
  \mlabel{lem:surject} 
If $\chi_1, \ldots, \chi_N \: \cA \to\C$ are mutually different characters 
of the complex algebra $\cA$, then the homomorphism 
\[ \chi \: \cA \to \C^N, \quad 
\chi(a) := (\chi_1(a),\ldots, \chi_N(a)) \] 
is surjective. 
\end{lem}

\begin{prf} This follows immediately from the fact that 
characters are linearly independent. 
\end{prf}

\begin{thm} \mlabel{thm:tenspro} 
Every bounded irreducible $*$-representation 
$(\pi, \cH)$ of $\g(\cA)$
is unitarily equivalent to a finite tensor product of irreducible evaluation representations;
there exists a finite set $\bx \subseteq \Gamma_{\cA}$
of involutive characters, and for each $\chi \in \bx$
an 
irreducible $*$-representation
$\rho_{\chi}$ of $\fg$, 
such that 
$\pi \simeq \pi_{\bx,\rho} := \bigotimes_{\chi \in \bx} \rho_{\chi} \circ \ev_{\chi}$. 
Conversely, all such representations are irreducible, 
and $\pi_{\bx,\rho} \simeq \pi_{\bx',\rho'}$ if and only if
$\bx = \bx'$ and $\rho_{\chi} \simeq \rho_{\chi}'$ for all $\chi\in \bx$.
\end{thm}

\begin{prf} In view of Theorem~\ref{thm:ind-charac}, 
we can write the highest weight $\lambda \: \fh \otimes \cA \to \C$ of 
$(\pi, \cH)$ in the form 
$\lambda = \sum_{\chi \in \bx} \lambda_\chi \otimes \chi$,
where the $\lambda_\chi$ are dominant weights and $\bx\subseteq \Gamma_{\cA}$ 
is a finite set of involutive characters of~$\cA$. 
Then 
\[ \ev_{\bx} \: \g(\cA) \to \g^{\bx}, \quad 
\ev(x \otimes a)(\chi) = \chi(a)x 
\] 
is a surjective homomorphism of Lie algebras (Lemma~\ref{lem:surject}) 
through which 
all the evaluation representations $\pi_{\chi,\rho_{\chi}}$ factor,
where $\rho_{\chi}$ is the representation with highest weight $\lambda_{\chi}$.
This implies that the tensor product 
$\pi_{\bx,\rho} = \bigotimes_{\chi \in \bx} \rho_{\chi} \circ \ev_{\chi}$ defines an irreducible 
representation of $\g(\cA)$, which clearly has highest weight 
$\lambda$. 

Conversely, suppose that $\pi_{\bx,\rho} \simeq \pi_{\bx',\rho'}$.
Then
$\ker \chi' \supset \bigcap_{\chi \in \bx} \ker \chi$ for all $\chi' \in \bx'$, so that 
$\chi' \in \bx$ by Lemma~\ref{lem:surject}.
Similarly, we have $\chi \in \bx'$ for all $\chi \in \bx$, 
whence $\bx' = \bx$.
It then follows from Lemma~\ref{findimeq} and the surjectivity of $\ev_{\bx}$  
that $\pi_{\chi} \simeq \pi'_{\chi}$ for all $\chi \in \bx$.
\end{prf}

Specializing to $\cA = C^\infty(X,\C)$ if $X$ is a compact manifold 
(Example~\ref{ex:cia}(a)) and to 
$\cA = C^\infty_c(X,\C)_+$ if $X$ is a non-compact manifold 
(Example~\ref{ex:cia}(b)),  we 
notice that $\Gamma_{\cA} \simeq X$, so that we obtain: 

\begin{cor} \mlabel{cor:tenspro} 
Let $X$ be a smooth manifold. 
Then every bounded irreducible $*$-representation 
$(\pi, \cH)$ of the Fr\'echet--Lie algebra 
$C^\infty(X,\fk)$ (if $X$ is compact) or 
$C^\infty_c(X,\fk) \rtimes \fk \cong \fk \otimes_\R C^\infty_c(X,\R)_+$ 
(if $X$ is non-compact) is a finite tensor product 
$\pi \cong \bigotimes_{x \in \bx} \rho_x \circ \ev_{x} $ 
of irreducible evaluation representations
for some finite subset $\bx \subseteq X$ and irreducible $*$-representations
$\rho_x$ of $\fg$.
Conversely, all such representations 
are irreducible,  
and $\pi_{\bx,\rho} \simeq \pi_{\bx',\rho'}$ if and only if
$\bx = \bx'$ and $\rho_{x} \simeq \rho_{x}'$ for \mbox{all $x \in \bx$}.
\end{cor}

\subsection{Translation to the group context}
Theorem \ref{thm:tenspro} classifies the bounded irreducible 
$*$-representations of $\fg(\cA)$.
We now discuss how these results lift to the group level, 
providing a complete classification 
of the bounded unitary representations of the $1$-connected 
Lie group $K_\cA$ with Lie algebra $\fk \otimes_\R \cA_\R$.
Here and throughout this section, $\cA$
will be a a unital commutative involutive cia and
$\fk$ a compact semisimple Lie algebra.

\subsubsection{Matrix groups over cias} 

We start by introducing some Lie groups related to~$\cA$. 

\begin{defn}  (a) Since $\cA$ is commutative, 
$\tr \:  \gl_n(\cA) \to \cA, (x_{ij}) \mapsto \sum_{j = 1}^n x_{jj}$ 
is a homomorphism of Lie algebras, so that 
\[ \fsl_n(\cA) := \{ x \in \gl_n(\cA) \: \tr x = 0\} \] 
is a closed ideal which is the Lie algebra of the Lie subgroup 
\[ \SL_n(\cA) := \ker(\det),\] 
where $\det \: \GL_n(\cA) \to \cA^\times$ is the natural determinant 
homomorphism (cf.\ \cite[Prop.~IV.3.4]{Ne06}). We write $\tilde\SL_n(\cA)$ 
for the unique $1$-connected locally exponential Lie group with 
Lie algebra $\fsl_n(\cA)$, which is the simply connected 
covering of the identity component $\SL_n(\cA)_0$. 

(b)  If $\cA$ is involutive, then the involution 
extends to all matrix algebras $M_n(\cA)$ by 
\[ (x_{ij})^* := (x_{ji}^*).\] 
We have corresponding unitary groups and their Lie algebras 
\[ \U_n(\cA) := \{ g \in \GL_n(\cA) \: g^* = g^{-1} \} \quad \mbox{ and } \quad 
\fu_n(\cA) := \{ x \in \gl_n(\cA) \: x^* = - x\}.\] 
The closed subalgebra $\su_n(\cA) := \fu_n(\cA)\cap \fsl_n(\cA)$ 
is the Lie algebra of the Lie subgroup 
\[ \SU_n(\cA) := \U_n(\cA) \cap \SL_n(\cA).\] 
Hence it is the Lie algebra of a unique $1$-connected Lie group, 
denoted $\tilde\SU_n(\cA)$. 
\end{defn} 

\begin{rem} \mlabel{rem:complexif} 
(a) Let $K$ be a $1$-connected compact Lie group  
and $G := K_\C$ its universal complexification. We write 
$\fk \subeq \fg = \fk_\C$ for their Lie algebras.  
Then $\fk_\cA = \fk \otimes_\R \cA_\R$ is a real form of the complex 
Lie algebra $\g(\cA) = \g \otimes_\C \cA$. 
Let $K_\cA$, resp., $G(\cA)$ be corresponding $1$-connected locally exponential 
Lie groups. Then the canonical morphism 
$\eta \: K_\cA \to G(\cA)$ for which 
$\L(\eta)$ is the inclusion $\fk \otimes_\R \cA_\R \into \fg \otimes_\C \cA$ 
has the following universal property. 
For each smooth morphism 
$\alpha \: K_\cA  \to H$, where $H$ is a complex Lie group with exponential 
function, there exists a unique holomorphic morphism 
$\alpha_\C \: G(\cA) \to H$ with 
$\alpha_\C \circ \eta = \alpha$. To verify this claim, we simply have to 
integrate the complex linear extension 
$\L(\alpha)_\C \: \g \otimes \cA  \to \L(H)$ to a group 
homomorphism (\cite[Thm.~4.1.19]{Ne06}). 

(b) For $\fk = \su_n(\C)$ and $\g = \fsl_n(\C)$ 
we have $K_\cA = \tilde\SU_n(\cA)$ and $G(\cA) = \tilde\SL_n(\cA)$. 
\end{rem}

\subsubsection{Bounded unitary representations} 

We now discuss the translation 
from Lie algebra to Lie group representations. 

\begin{defn} Let $G$ be a complex locally exponential Lie group endowed 
with an antiholomorphic antiautomorphism $g \mapsto g^*$.  
A {\it holomorphic $*$-representation of $G$} is a pair 
$(\pi, \cH)$, where $\cH$ is a complex Hilbert space and 
$\pi \: G \to \GL(\cH)$ is a holomorphic homomorphism 
satisfying $\pi(g^*) = \pi(g)^*$ for $g \in G$. 
Then $\dd\pi \: \g \to B(\cH)$ is a bounded $*$-representation of $\g$ 
in the sense of Definition~\ref{def:bsrep}. 
\end{defn}

With Remark~\ref{rem:complexif} we immediately obtain the following generalization 
of Weyl's correspondence between unitary representations of $K_\cA$ and holomorphic 
representations of $G(\cA)$: 

\begin{prop} Let $\cA$ be a commutative involutive cia, 
$\g$ a semisimple complex Lie algebra with compact real form $\fk$, 
$G(\cA)$ a $1$-connected Lie group with Lie algebra 
$\g \otimes_\C \cA$, and 
$K_\cA$ a $1$-connected Lie group with Lie algebra 
$\fk \otimes_\R \cA_\R$. 
 
If $(\pi_\C, \cH)$ is a holomorphic $*$-representation of $G(\cA)$, then 
$\pi := \pi_\C \circ \eta$ is a bounded unitary representation of 
$K_\cA$. 
We thus obtain a one-to-one correspondence between holomorphic $*$-representations 
$\pi_\C$ of $G(\cA)$ and bounded unitary representations $\pi$ of $K_\cA$. 

The commutants $\pi(K_\cA)'$ and 
$\pi_\C(G(\cA))'$ coincide, so that $\pi$ is irreducible 
if and only if $\pi_\C$ has this property. 
\end{prop}

Combining Theorem~\ref{thm:tenspro} with the preceding 
proposition, we obtain: 

\begin{thm} \mlabel{thm:6.10} 
 Let $\cA$ be a commutative involutive cia. 
Then every bounded irreducible unitary representation $(\pi, \cH)$ 
of $K_\cA$ is a finite tensor product of evaluation representations 
corresponding to irreducible representations of $K$. 
In particular, $\cH$ is finite-dimensional. 
\end{thm}

\begin{rem} Let $K$ be a compact Lie group with Lie algebra $\fk$. 
Then Tychonov's Theorem implies that the product group $K^X$ of 
{\em all} maps $X \to K$ is a compact group. 
Any irreducible continuous unitary representation $(\pi, \cH)$ of 
$K^X$ is finite-dimensional and factors through a projection to 
some quotient Lie group $K^\bx$, where $\bx \subeq X$ is a finite subset. 
Hence there exist irreducible unitary representations 
$(\rho_x, V_x)$, labeled by $x \in \bx$, such that 
$\pi(g) = \otimes_{x \in \bx} \rho_x(g_x)$. 

It now follows from Corollary~\ref{cor:tenspro} that
every 
irreducible bounded unitary representation of the 
connected Lie group $C^\infty(X,K)_0$, where $X$ is a compact manifold, 
extends to a continuous representation of the compact group 
$K^X$. This observation may be of some interest in the context
 of Loop Quantum Gravity where one works with compactified 
gauge groups of the form $K^X$ (cf.\ \cite{Th08}). 
As we shall see in Section~\ref{sec:4} below, 
this picture changes for the 
group $C^\infty_c(X,K)$ when $X$ is a non-compact manifold. In this case 
there exist bounded irreducible unitary representations that do not 
extend to the compact group $K^X$.  However, it turns out that they all factor through 
representations of groups of the type $K^{(\bx)}$
with $\bx \subset X$ a \emph{locally} finite subset. 

\end{rem}

\section{Irreducible \texorpdfstring{$*$}{*}-representations of \texorpdfstring{$\fsl_2(\cA)$}{sl2(A)}} 

\mlabel{sec:3}

In this section we apply Theorem~\ref{thm:2.5} on multiplicative characters 
to obtain a classification of 
the inducible functionals of
$\fsl_2(\cA)$,
where $\cA$ is a unital involutive commutative cia. 
We have already seen how this was used in Theorem~\ref{thm:ind-charac} 
to 
obtain the corresponding result for tensor products 
$\g(\cA) = \fk_\C \otimes \cA$, where $\fk$ is a compact Lie algebra and 
$\g = \fk_\C$. Theorem~\ref{thm:ind-charac} in turn led to the characterization
of bounded irreducible $*$-representations in 
Theorem~\ref{thm:tenspro}. 


Throughout this section, $\cA$ will be a unital commutative involutive cia.

\subsection{\texorpdfstring{The group $\tilde\SL_2(\cA)$}{The group SL2(A)}} 

We use the standard notation for the basis elements 
of $\fsl_2(\C)$: 
\[ h := \pmat{1 & 0 \\ 0 & -1}, \quad 
e := \pmat{0 & 1 \\ 0 & 0} \quad \mbox{ and } \quad 
f := \pmat{0 & 0 \\ 1 & 0} \] 
satisfying the relations 
$[h,e] = 2e$,
$[h,f] = -2f$ and 
$[e,f] = h$.
In the notation of Section~\ref{sec:2}, we have 
$\g = \fg^+ \oplus \g^0 \oplus \fg^-$ with\footnote{Recall that 
$\fg^+$ is the span of the \emph{negative} roots, cf. the footnote 
\ref{note1}.}
\[ \g = \fsl_2(\C), \quad 
\g^0 = \cA h, \quad 
\fg^- = \cA e \quad \mbox{ and } \quad 
\fg^+ = \cA f.\] 

Using Gauss decomposition, we see that the 
identity component $\SL_2(\cA)_0$ of the Lie group $\SL_2(\cA)$
is generated by matrices of the form 
\[
\pmat{1 & a \\ 0 & 1}, \quad \pmat{1 & 0 \\ b & 1},\quad a,b \in \cA\,. 
\]
Indeed, this follows from 
\[ \pmat{ a & b \\ c & d} \in e^{\cA f} \Delta(a) e^{\cA e}\quad \mbox{ for } \quad 
\Delta(a) := \pmat{a & 0 \\ 0 & a^{-1}} \quad \mbox{ and } \quad a \in \cA^\times\]  
and 
\begin{equation}
  \label{eq:diagelt}
\Delta(a) = \pmat{a & 0\\ 0 & a^{-1}}= \pmat{1 & 0 \\ a^{-1} - 1 & 1} 
\pmat{1 & 1 \\  0  & 1} 
\pmat{1 & 0 \\  a - 1 & 1} 
\pmat{1 & - a^{-1}\\ 0  & 1} \quad \mbox{ for } \quad 
a \in \cA^\times. 
\end{equation}
We write $q \: \tilde \SL_2(\cA) \to \SL_2(\cA)_0$ for the universal 
covering homomorphism  with $\L(q) = \id_{\fsl_2(\cA)}$ and 
$e_{12} \: (\cA,+) \to \tilde \SL_2(\cA)$ 
for the unique continuous homomorphism satisfying
\[
q(e_{12}(x)) = \pmat{1 & x \\ 0 & 1} \quad \mbox{ for } \quad x \in \cA.
\]
We likewise define $e_{21} \: \cA \to \tilde \SL_2(\cA)$. 
For $a \in \cA^\times$, we define 
$\tilde \Delta(a) \in \tilde \SL_2(\cA)$ by 
\begin{equation}
  \label{eq:tilded}
\tilde\Delta(a):= e_{21}(a^{-1}-1) e_{12}(1) e_{21}(a-1) e_{12}(-a^{-1})
\end{equation}
and observe that $\tilde \Delta(\1)= \1$.
In view of \eqref{eq:diagelt}, 
we have 
$q \circ \tilde \Delta =  \Delta$ on  $\cA^\times$. 

This means that the restriction $\tilde \Delta \: \cA^\times_0 \to \tilde \SL_2(\cA)$ to the identity 
component $\cA^\times_0$ is the unique 
continuous lift of the homomorphism $\Delta \: \cA^\times_0 \to \SL_2(\cA)_0$ 
to the simply connected covering group $\tilde\SL_2(\cA)$, 
hence in particular a morphism of Lie groups (cf.\ \cite{GN13}). This in turn 
implies that 
\begin{equation}
  \label{eq:exprel}
\tilde \Delta(e^a) = \exp_{\tilde \SL_2(\cA)}(ah) \quad \mbox{ for } \quad a \in \cA,
\end{equation}
where $e^x = \sum_{n = 0}^\infty \frac{x^n}{n!}$ 
is the exponential function of the Lie group 
$\cA^\times$ and $\exp_{\tilde \SL_2(\cA)}$ is the exponential function of the Lie group 
$\tilde \SL_2(\cA)$.


\subsection{Inducible functionals on \texorpdfstring{$\fsl_2(\cA)$}{sl2(A)}} 

Let $(\pi, \cH)$ be a bounded irreducible $*$-representation of 
$\fsl_2(\cA)$. We consider the closed subspace 
\[ \cE := \cH^{\fg^-} = \cH^{\cA e} = \ker(\pi(\cA e))\] 
and recall from Proposition~\ref{prop:6.1} that 
the representation 
$(\rho, \cE)$ of $\g^0 = \cA h$ is one-dimensional, hence can 
be written as 
$\rho(ah) = \lambda(a) \1$  for some $\lambda \in \cA'$. 
In this subsection we shall determine which linear functionals 
arise from this construction. 

To simplify notation, we now  put 
\[ G := \tilde\SL_2(\cA) \quad \mbox{ and } \quad G^0 := Z_G(h)\,.\]
In order to determine which $\lambda \in \cA'$ have the above property, 
we need to 
go to the level of Lie groups, i.e.\ we need 
the
holomorphic $*$-representation 
$\pi_{G} \: \tilde \SL_2(\cA) \to \GL(\cH)$ 
with $\dd \pi_{G} = \pi$ (cf.\ Remark~\ref{rem:complexif}). 

Let $P \in B(\cH)$ denote the orthogonal projection to $\cE$. 
Then 
\[ \phi \: G \to B(\cE), \quad 
\phi(g) := P \pi_G(g) P \]  
is a holomorphic function. 
Observe that the homomorphism 
$\tilde \Delta \: \cA^\times_0 \to G$ defines an isomorphism 
$  
\tilde \Delta \: \cA^\times_0 \to (G^0)_0
$
of locally exponential complex Lie groups. 


\begin{lem} \mlabel{lem:4.3} 
With $G$, $\phi$ and $\rho$ as above, we have:
  \begin{itemize}
    \item[\rm(i)] $\phi(g_- g g_+) = \phi(g)$ for $g \in G$, 
$g_+ \in e_{12}(\cA)$, $g_- \in e_{21}(\cA)$. 
 \item[\rm(ii)] $\rho_G = \phi\res_{G^0}  \: G^0 \to \GL(\cE)$ is a representation 
with $\dd\rho_G = \rho$. 
    \item[\rm(iii)] $\phi(\tilde \Delta(e^x)) = \rho_G(\exp(xh))$ for 
$x \in \cA$. 
    \item[\rm(iv)] The group homomorphism $\rho_G \circ \tilde \Delta \: \cA^\times_0 \to 
\GL(\cE)$ 
extends to a polynomial function $F \: \cA \to B(\cE)$. 
  \end{itemize}
\end{lem}

\begin{prf}We prove the above point by point.
\begin{itemize}
 \item[(i)]
 By definition, the elements of $\cE$ are fixed by 
$e_{12}(\cA)$, so that $\phi(gg_+) = \phi(g)$ for $g \in G$ and $g_+ \in e_{12}(\cA)$. 
From $\phi(g)^* = \phi(g^*)$ and $e_{12}(\cA)^* = e_{21}(\cA)$, we now have (i).
 \item[(ii)]
 As $\pi_G(G^0)$ normalizes $\fg^-$, it preserves $\cE$, and this 
proves (ii).
\item[(iii)]
follows from (ii) and \eqref{eq:exprel}.
\item[(iv)]
From (i) and the definition of $\tilde \Delta$, we derive 
\[ \phi(\tilde \Delta(a)) 
= \phi(e_{12}(1) e_{21}(a-1)) 
=P \pi_G(e_{12}(1) e_{21}(a-1)) P 
=P  e^{\pi(e)} e^{\pi((a-\1)f)}P.\] 
That this function is polynomial in $a$ 
follows from the existence of a natural number $N \in \N$ with 
${\pi(\cA f)^{N} = \{0\}}$ 
(Proposition~\ref{prop:6.1}(iii)). 
\end{itemize}
\end{prf}

\begin{prop} \mlabel{prop:6.4} 
For $\fsl_2(\cA)$, any inducible functional 
$\lambda \: \cA \to \C$ is a finite sum of characters. 
Any finite sum of involutive characters is inducible. 
\end{prop}

\begin{prf} Since $F$ is a polynomial map to $B(\cE) \cong \C$ 
and $F(ab) = F(a) F(b)$ holds for 
$a,b \in \cA^\times_0 = e^\cA$, analytic continuation implies that 
$F(ab) = F(a) F(b)$ for $a,b \in \cA.$
On the other hand, $F(e^x) = e^{\lambda(x)}$ for $x \in \cA$. 
Now Corollary~\ref{cor:4.9} implies the existence 
of $\chi_1, \ldots, \chi_N \in \Gamma_\cA$ with 
$F = \prod_{j = 1}^N \chi_j$. Differentiating in $\1$, we obtain  
\[ \lambda = \chi_1 + \ldots + \chi_N.\] 
For the converse, we only have to show that any involutive 
character $\chi$ is inducible (cf. Lemma~\ref{lem:sum}). This follows from the fact that 
$\ev_\chi \: \fsl_2(\cA) \to \fsl_2(\C) \subeq \gl_2(\C)$ is a 
$2$-dimensional $*$-representation with 
highest weight $\lambda = \chi$. 
\end{prf}

\begin{rem} \mlabel{rem:non-invol} 
Every inducible character is involutive.
Indeed, suppose that
$\chi \: \cA \to \C$ is a non-involutive character. 
The pair $(\chi, \chi^*)$ then defines a surjective homomorphism 
\[ \ev \: \fsl_2(\cA) \to \fsl_2(\bD) \cong \fsl_2(\C) \oplus \fsl_2(\C), \] 
where the involution on $\bD = \C^2$ is given by $(z_1, z_2)^* = (\oline{z_2}, \oline{z_1})$ 
(cf.\ Remark~\ref{rem:bd} and Lemma~\ref{lem:surject}). 
Then $\fsl_2(\bD) \cong \fsl_2(\C)^2$, but the corresponding real form is 
\[ \su_2(\bD) = \su_2(\C) \otimes_\R \bD_\R 
= (\su_2(\C) \otimes (1,1)) \oplus 
 (\su_2(\C) \otimes (i,-i)) 
= \{ (z, - z^*) \: z \in \fsl_2(\C) \}. \] 
As a real Lie algebra, we thus obtain 
$\su_2(\bD) \cong \fsl_2(\C)$, which is a simple real Lie algebra with 
no non-zero bounded unitary representation. This means that 
$\chi$ is not inducible.
\end{rem}

The following theorem closes the gap in the characterization of 
inducible functionals in Proposition~\ref{prop:6.4}. 

\begin{thm} \mlabel{thm:6.x} For $\fsl_2(\cA)$, any inducible functional 
$\lambda \: \cA \to \C$ is a finite sum of involutive algebra characters. 
Conversely, any such sum is inducible. 
\end{thm}

\begin{prf} In view of Proposition~\ref{prop:6.4}, it only remains to show that,  
if a functional $\lambda$ is inducible and a finite sum of characters, 
then it can be written as a finite sum of involutive characters. 

Let $(\pi_\lambda, \cH_\lambda)$ be the irreducible $*$-representation 
of $\fsl_2(\cA)$ with highest weight $\lambda$ and unit highest weight 
vector $v_\lambda$ spanning $\cE_\lambda$. 
According to
Proposition~\ref{prop:6.4}, we then have
\[ \lambda = \chi_1 + \cdots + \chi_N \quad \mbox{ with } \quad 
\chi_j \in \Gamma_\cA\,.\]   
We rewrite 
\[ \lambda = m_1 \chi_1 + \cdots + m_k \chi_k, \quad \chi_j \in \Gamma_\cA, 
m_j \in \N_0, \] 
where the $\chi_j$ are pairwise 
different, hence linearly independent in the dual space $\cA'$. 
Since $\lambda = \lambda^*$, the relation 
\[ m_1 \chi_1 + \cdots + m_k \chi_k=  m_1 \chi_1^* + \cdots + m_k \chi_k^* \] 
implies the existence of an involution 
$\sigma\in S_k$ with
$\chi_j^* = \chi_{\sigma(j)}$ for $j =1,\ldots, k$. 
If $\sigma(j) = j$, then $\chi_j$ is involutive;  
if $\sigma(j) \not= j$, then $\chi_j^* = \chi_{\sigma(j)}$. 
It follows in particular that $m_j = m_{\sigma(j)}$. We may thus write 
\[ \lambda = 
\sum_{j = 1}^\ell a_j \eta_j + 
\sum_{j = 1}^n b_j (\gamma_j + \gamma_j^*), \quad 
a_j, b_j \in \N_0,\] 
where the $\eta_j$ are involutive and $\gamma_j \not=\gamma_j^*$. 
This shows in particular that the ideal 
\begin{equation}
  \label{eq:ideal}
 \cI := \bigcap_{j = 1}^N \ker \chi_j  
= \bigcap_i \ker \eta_i \cap \bigcap_j (\ker \gamma_j \cap \ker \gamma_j^*) 
\end{equation}
is $*$-invariant. 
Since $e\otimes \cI \subseteq \fg^-$, the Lie algebra
$e \otimes \cI + h \otimes \cI$ annihilates $v_\lambda$. This 
implies that the linear functional 
$\alpha(X) := \la \pi_\lambda(X)v_\lambda, v_\lambda \ra$ satisfies 
$\fsl_2(\cI) \subeq \ker \alpha.$ 
Next we observe that, for $a \in \cI$, 
\[ \| \pi_\lambda(f \otimes a)v_\lambda\|^2 
= \la \pi_\lambda(f \otimes a)^*\pi_\lambda(f \otimes a) v_\lambda, v_\lambda \ra 
= \la \pi_\lambda([(f \otimes a)^*, f \otimes a]) v_\lambda, v_\lambda \ra 
\in \alpha(\fsl_2(\cI)) = \{0\},\] 
whence 
$v_\lambda \in \cH_\lambda^{\fsl_2(\cI)}.$ 
Since $\fsl_2(\cI) \trile \fsl_2(\cA)$ is an ideal, 
the subspace $\cH_\lambda^{\fsl_2(\cI)}$ is invariant under 
$\fsl_2(\cA)$. As it contains the cyclic vector $v_\lambda$, 
it follows that $\fsl_2(\cI) \subeq \ker \pi_\lambda$ 
(cf.\ \cite[Lemma~IX.1.3]{Ne00} for similar arguments). 
Therefore the representation $\pi_\lambda$ factors 
through a representation $\oline\pi_\lambda$ of the 
involutive quotient algebra $\fsl_2(\cA/\cI)$. 
We now have to understand the structure of this algebra. 

With \eqref{eq:ideal} and Lemma~\ref{lem:surject}, we see that 
$\codim \cI = \ell + 2 n$. Accordingly, 
$\cA/\cI$ is isomorphic to the algebra 
$\C^\ell \oplus \bD^n$ with the involution 
\[ (x_1, \ldots, x_\ell, y_1, \ldots, y_n, z_1, \ldots, z_n)^* 
= (\oline{x_1}, \ldots, \oline{x_\ell}, \oline{z_1}, \ldots, \oline{z_n}, 
\oline{y_1}, \ldots, \oline{y_n}).\] 
We thus obtain 
\[ \su_2(\cA/\cI) \cong \su_2(\C)^\ell \oplus \fsl_2(\C)^n\]  
(cf.\ Remark~\ref{rem:non-invol}). 
Since $\fsl_2(\C)$ has no non-zero bounded $*$-representations, 
$\oline\pi_\lambda$ is trivial on the corresponding 
factors, and this in turn implies that $n = 0$. 
\end{prf}

\section{Lie algebras of smooth sections} 

\mlabel{sec:4} 

Building on Corollary~\ref{cor:tenspro}, 
we now extend our classification results to 
Lie algebras of smooth sections of Lie algebra bundles $\fK \to X$,
where the typical fiber $\fk$ of $\fK$ is compact semisimple, 
and $X$ is a 
$\sigma$-compact smooth manifold with compact boundary 
$\partial X$. 
This includes in particular the Fr\'echet--Lie algebras   
$\gau(P)$ of infinitesimal gauge transformations of 
principal bundles $P \to X$ with compact semisimple structure group~$K$. 

\subsection{Lie algebra bundles}

Let $q \: \fK \to X$ be a smooth 
Lie algebra bundle over 
$X$ 
whose typical fiber $\fk$ is a finite-dimensional Lie algebra. 
Let $\Gamma(\fK)$ denote the space of 
sections $s \: X \to \fK$ that 
are smooth on the interior $X^\circ$, and whose derivatives extend 
continuously to the boundary.
We endow $\Gamma(\fK)$  
with the smooth compact open topology obtained from the embedding 
\[ \Gamma(\fK) \into \prod_{n \in \N_0} C(T^n X, T^n \fK), \quad 
s \mapsto (T^n s)_{n \in \N_0},\] 
and the compact open topology on the spaces $C(T^n X, T^n \fK)$. 
This turns $\Gamma(\fK)$ into a Fr\'echet--Lie algebra with respect to the 
pointwise bracket 
\[ [s_1, s_2](x) := [s_1(x), s_2(x)]\]  
(cf.\ \cite[Thm.~II.2.7]{Ne06}). 

We write $X = \bigcup_{n \in \N} X_n$, where $X_n$ is a compact submanifold 
with boundary,
$X = \bigcup_n X_n$ and $X_n \subeq X_{n+1}^0$ for 
$n \in \N$. (This is possible because we required $\partial X$ to be 
compact.)
We also put $\fK_n := \fK\res_{X_n}$ and observe that the 
restriction map $r_n \: \Gamma(\fK) \to \Gamma(\fK_n)$ is surjective 
for every $n$ (cf.~\cite{Wo06}).
Therefore the embedding
\[ \Gamma(\fK) \to \prolim \Gamma(\fK_n) \] 
is a continuous bijective linear map between Fr\'echet spaces, 
hence a topological isomorphism by the Open Mapping Theorem 
(cf.\ \cite{Ru91}). 

The space $\Gamma_c(\fK)$ of compactly supported smooth functions is 
the union of the closed ideals 
\[ \Gamma(\fK)_{X_n}  := \{ s \in \Gamma(\fK) \:  \supp(s) \subeq X_n\} 
\trile \Gamma(\fK) \] 
which are Fr\'echet spaces. We endow 
$\Gamma_c(\fK) \cong \indlim \Gamma(\fK)_{X_n}$ with the 
corresponding locally convex direct limit topology 
which turns it into an LF-Lie algebra, i.e., an LF-space with a 
continuous Lie bracket (\cite[Cor.~F.24, Rem.~F.28]{Gl04}). 

The following proposition reduces the problem of describing 
the bounded unitary representations of $\Gamma(\fK)$ to the 
case where $X$ is a compact manifold with boundary. 

\begin{prop} \mlabel{prop:6.1b} 
For every bounded unitary representation 
$(\pi, \cH)$ of the Fr\'echet--Lie algebra 
$\Gamma(\fK)$, there exists a compact submanifold $Y \subeq X$ 
with boundary and a bounded unitary representation
$(\oline\pi, \cH)$ of $\Gamma(\fK\res_Y)$ such that 
$\pi(s) = \oline\pi(s\res_Y)$ for every $s \in \Gamma(\fK)$. 
\end{prop}

\begin{prf} Since $\Gamma(\fK)$ is the projective limit of the 
Fr\'echet spaces $\Gamma(\fK_n)$, there exists an $n \in \N$ 
and a continuous seminorm $p$ on $\Gamma(\fK_n)$ such that 
\[ \|\pi(s)\| \leq p(s\res_{X_n}) \quad \mbox{ for } \quad 
s \in \Gamma(\fK).\] 
This implies that $\pi$ vanishes on the kernel of the restriction 
map $r_n$, and since $r_n$ is a quotient map, the assertion follows 
with $Y = X_n$. 
\end{prf}

\begin{rem} \mlabel{rem:4.2} 
Suppose that $\fk$ is compact. 
Let $\fk = \fz(\fk) \oplus [\fk,\fk]$ denote the decomposition of 
$\fk$ into  center and the 
semisimple commutator algebra. Since this decomposition 
is invariant under the full automorphism group $\Aut(\fk)$, it 
induces a direct sum decomposition 
$\fK \cong Z(\fK) \oplus [\fK,\fK]$ of Lie algebra bundles, which in turn leads to 
\[  \Gamma(\fK) \cong \Gamma(Z(\fK)) \oplus \Gamma([\fK,\fK]).\] 

If $(\pi, \cH)$ is a bounded factor representation of 
$\Gamma(\fK)$, then 
$\pi(\Gamma(Z(\fK))) \subeq Z(\pi(\Gamma(\fK))'') = \C \1$. 
Therefore the representation $\pi$ is a tensor product 
of a one-dimensional unitary representation 
of the abelian Lie algebra $\Gamma(Z(\fK))$ and a 
factor representation of $\Gamma([\fK,\fK])$. 
Since every continuous linear map 
$\lambda \: \Gamma(Z(\fK)) \to i \R$ defines a one-dimensional 
unitary representation, the classification of 
 bounded factor representations of $\Gamma(\fK)$ reduces to the 
corresponding problem for $\Gamma([\fK,\fK])$.   
\end{rem}

\begin{ex} (a) Typically, Lie algebra bundles 
arise as $\fK := \Ad(P)$ for a smooth $K$-principal bundle $q \: P \to X$, where 
$K$ is a  Lie group with Lie algebra $\fk$. 
The adjoint bundle $\Ad(P) := P \times_{\Ad}\fk \rightarrow X$ 
is the orbit space of the action $K \curvearrowright P \times \fk$ 
defined by
$k\cdot(p,x) := (pk^{-1}, \Ad(k)x)$. 

  
The group of vertical bundle automorphisms of $P$ is called 
the \emph{gauge group} $\Gau(P)$.
It is a locally convex Lie group if $M$ is compact.
Each gauge transformation 
$g \in \Gau(P)$ is of the form $g(p) = p\tilde g(p)$
with $\tilde{g} \in C^\infty(P,K)^K$, 
i.e.\ $\tilde{g}(pk) = k^{-1}\tilde{g}(p)k$ for all
$p \in P, k \in K$.
The map $g \mapsto \tilde{g}$ is
an isomorphism of locally exponential Lie groups 
$ \Gau(P) \to C^\infty(P,K)^K$. 
Accordingly, we obtain an isomorphism of Lie algebras
\[ \gau(P) = \Gamma(\Ad(P)) \to 
C^\infty(P,\fk)^K := \{ f \in C^\infty(P,\fk)\: (\forall p \in P, k \in K)\, 
f(pk) = \Ad(k)^{-1}f(p)\}.\]  

(b) If $\rho \: K \to \U(V)$ is a continuous finite-dimensional unitary 
representation of $K$, then we obtain an associated 
vector bundle $\bV := P \times_\rho V$
as the orbit space of the action $K \curvearrowright P \times V$ 
defined by $k\cdot(p,v) := (pk^{-1}, \rho(k)v)$. 
As in the case of the adjoint bundle, one identifies
sections $s\in \Gamma(\bV)$ with equivariant functions
$\tilde s \in C^{\infty}(P,V)^K$.

The gauge group $\Gau(P)$ acts on $\bV$ by bundle automorphisms via 
\[ g\cdot[p,v] := [g(p),v] = [p\tilde g(p), v] 
= [p, \rho(\tilde g(p))v].\] 

For any Radon measure 
$\mu$ on $X$, we obtain on the space $\Gamma(\bV)$ of smooth 
sections of $\bV$ a scalar product by 
\[ \la s, t \ra := \int_X \la s(x), t(x)\ra \, d\mu(x).\] 
On the Hilbert completion $\Gamma^2(\bV,\mu)$ of $\Gamma(\bV)$,
this yields
a unitary 
representation of $\Gau(P)$ by 
$ (g\cdot s)(x) := g\cdot s(x)$.
If $s$ is identified with $\tilde s \in C^{\infty}(P,V)^{K}$,
this reads
$
(g.\tilde s)(p) = \rho(\tilde g(p)) \tilde s(p)
$. 
If $X$ is compact, then this representation 
$\Gau(P) \to \U(\Gamma^2(\bV,\mu))$ is norm continuous, and the corresponding 
derived representation 
\[ (\xi.\tilde s)(p) = \dd\rho(\xi(p)) \tilde s(p) \] 
is a bounded unitary representation of $\gau(P)$. 

Since the commutant of this representation always contains the multiplications 
with elements of $L^\infty(X,\mu)$, it is irreducible if and only if 
$V$ is irreducible, and
$\mu$ is non-zero and supported in a single point $x_0 \in X$. 
Fix $p_0 \in P$ with $q(p_0) = x_0$. Then 
$\Gamma^2(\bV,\mu) \to V, s \mapsto \tilde s(p_0)$ 
is a unitary equivalence intertwining the representation of 
$\Gau(P)$ on $\Gamma^2(\bV,\mu)$ with the evaluation representation on $V$ by 
$\pi_{p_0}(g) := \rho(\tilde g(p_0))$.
\end{ex}

\subsection{Local structure of irreducible bounded representations}
We proceed with the classification of irreducible bounded
unitary representations in terms of \emph{evaluation representations}.
We aim to prove that they are tensor products of evaluation representations. 
In order to do this, we investigate the 
local structure of bounded unitary representations.

\begin{defn} Let $x \in X$ 
and let $(\rho, V)$ be a bounded representation 
of $\fK_x \cong \fk$. Then $\pi(s) := \rho(s(x))$ defines a 
bounded unitary representation of $\Gamma(\fK)$. 
We call these representations {\it evaluation representations}. 
Note that $\pi$ is irreducible if and only if $\rho$ is.
\end{defn}

\begin{rem} \mlabel{rem:decomp} 
In view of the paracompactness of $X$, there exists a 
locally finite open covering $(U_j)_{j \in J}$ 
by  relatively compact subsets $U_j \subeq X$ for which 
$\fK$ is trivial on an open neighborhood of $\oline{U_j}$. 
Then the space $\Gamma_c(\fK\res_{U_j})$ 
of sections of $\fK$ with support contained in $U_j$ is isomorphic to 
$C_c^\infty(U_j,\fk)$ and a partition of unitary argument 
shows that 
\[ \Gamma_c(\fK) = \sum_{j \in J} \Gamma_c(\fK\res_{U_j}),\] 
where the $\Gamma_c(\fK\res_{U_j})$ are ideals isomorphic to $C^\infty_c(U_j,\fk)$.  
\end{rem}

\begin{lem}
  \mlabel{lem:c.20} If $\fk$ is perfect, i.e.\ $\fk = [\fk,\fk]$, then also 
$\Gamma_c(\fK)$ is perfect.
\end{lem}

\begin{prf} Since $\Gamma_c(\fK)$ is a sum of subalgebras of the form 
$C^\infty_c(X,\fk)$ (Remark~\ref{rem:decomp}), 
it suffices to show that $C_c^\infty(X,\fk) \simeq C_c^\infty(X,\R)\otimes_{\R} \fk$ 
is perfect for 
every smooth manifold $X$. 
Since $\fk$ is perfect, every $x \in \fk$ can be written as 
$x = \sum_{j = 1}^k [y_j, z_j]$ with $y_j, z_j \in \fk$. 
For $f \in C_c^\infty(X,\R)$, we choose a function 
$\chi \in C_c^\infty(X,\R)$ with $\chi\res_{\supp(f)} = 1$. Then 
\[ \sum_{j = 1}^k [y_j \otimes \chi, z_j \otimes f] 
=  \sum_{j = 1}^k [y_j,z_j] \otimes \chi f  = x \otimes f\] 
shows that $C_c^\infty(X,\fk)$ is perfect.
\end{prf}
\noindent Note that the above lemma applies in particular to compact semisimple Lie algebras $\fk$,
which are automatically perfect.

\begin{lem} \mlabel{lem:c.21} 
Let $(\rho, \cH)$ be a finite tensor product of 
irreducible evaluation representations
at different points for 
an ideal $\Gamma_c(\fK\res_{U})$ ($U \subeq X$ open) of 
$\Gamma_c(\fK)$. 
Then $(\rho, \cH)$
extends uniquely to a bounded unitary representation 
$(\oline\rho, \cH)$ of $\Gamma_c(\fK)$ 
on the same space. It is again a finite tensor product 
of irreducible evaluation representations
at different points
\end{lem}

\begin{prf} Since $\rho$ is a finite tensor product of evaluation
 representations, the existence of the extension 
follows from the obvious extensions of evaluation representations. 

To see that the extension is unique, 
note that a finite tensor product of 
irreducible evaluation representations
is itself irreducible, provided that one evaluates at 
different points $x$ of $X$.
Now suppose that 
$\tilde\rho$ and $\oline\rho$ are two extensions of $\rho$. 
Then, for each $x \in \Gamma_c(\fK)$, the operator 
$\tilde\rho(x) - \oline\rho(x)$ commutes with 
$\rho(\Gamma_c(\fK\res_{U_j}))$, so that Schur's Lemma implies that it is of the form 
$\alpha(x) \1$ for some $\alpha(x) \in i \R$. 
Then $\alpha \: \Gamma_c(\fK) \to \R$ is a one-dimensional representation, 
hence vanishes on all brackets. As $\Gamma_c(\fK)$ is perfect by Lemma~\ref{lem:c.20}, 
$\alpha =0$, and therefore $\tilde\rho = \oline\rho$. 
\end{prf}

\begin{lem} \mlabel{lem:facrep} 
Let $\fg$ be a Lie algebra and $\fn \trile \g$ be an ideal. 
Suppose that $\pi \: \g \to \fu(\cH)$ is a unitary factor 
representation. Then
\begin{description}
\item[\rm(i)] $\pi\res_\fn$ is a factor representation. 
\item[\rm(ii)] If $\pi\res_\fn$ is a type I representation, i.e., 
a multiple of an irreducible representation $(\rho, V)$, and if
$\rho$ extends to a bounded irreducible representation 
$(\pi_2, \cH)$ of $\g$, then 
there exists a bounded representation 
$(\pi_1, \cH_1)$ of $\g$ 
such that $\pi \cong \pi_1 \otimes \pi_2$ and $\fn \subeq \ker \pi_1$. 
Then $\pi_1$ is a factor representation which is 
irreducible if and only if $\pi$ is irreducible. 
\end{description}
\end{lem}

\begin{prf} (i) 
Let $\cM := \pi(\fn)'' \subeq B(\cH)$ denote the 
bicommutant of $\pi(\fn)$. The fact that $\fn \trile \g$ is an ideal 
implies that $\cM$ is invariant under $\ad(\pi(\g))$, so that 
we obtain for each $x \in \g$ a derivation $\ad(\pi(x))$ of $\cM$. 
Since every derivation of a von Neumann algebra is inner 
(\cite[Thm.~XI.3.5]{Ta03}), $\ad(\pi(\g))$ annihilates the 
center of $\cM$, so that $Z(\cM) \subeq \pi(\g)'$. 
On the other hand $Z(\cM) \subeq \cM \subeq \pi(\g)''$, so that 
$Z(\cM) \subeq Z(\pi(\g)'') = \C \1$ since $\pi(\g)''$ is a factor. 
We conclude that $\cM$ is also a factor. 

(ii)  If $\cM$ is of type I, 
there exists an irreducible bounded unitary representation 
$(\rho, \cH_2)$ of $\fn$ and a Hilbert space $\cH_1$ (the multiplicity 
space) such that $\cH \cong {\cH_1} \hat\otimes \cH_2$ and 
$\pi\res_\fn = \1_{\cH_1} \otimes \rho$. 
Extending the irreducible representation $\rho$ to a representation 
$\pi_2$ of $\g$ on $\cH_2$, we obtain the representation 
$\tilde\pi_2 := \1 \otimes \pi_2$ of $\g$ on $\cH = {\cH_1} \hat\otimes \cH_2$ 
which coincides on $\fn$ with $\pi$. For each $x \in \g$, the operator 
\[ \tilde\pi_1(x) := \pi(x) - \tilde\pi_2(x) \] 
commutes with $\pi(\fn)$ with generates the von Neumann algebra~$\cM 
\cong \1 \otimes B(\cH_2)$. Therefore 
\[ \tilde\pi_1(x) = \pi_1(x) \otimes \1 \in B({\cH_1}) \otimes \1_{\cH_2} 
\quad \mbox{ for some } \pi_1(x) \in B(\cH_1),\] 
which implies in particular that 
$\tilde\pi_1(x)$ commutes with $\tilde\pi_2(\g)$. This leads to 
\[ \pi(x) = \pi_1(x) \otimes \1 + \1 \otimes \pi_2(x),\] 
i.e., $\pi \cong \pi_1 \otimes \pi_2$. Now 
\[ \pi(\g)' \subeq \pi(\fn)' = (\1 \otimes \cM)' = B(\cH_1) \otimes \1 \] 
leads to 
\[ \pi(\g)' = \pi(\g)'\cap (B(\cH_1) \otimes \1) 
= \pi_1(\g)' \otimes \1 \cong \pi_1(\g)',\] 
and further to 
\[ \pi(\g)'' = \pi_1(\g)'' \otimes B(\cH_2).\] 
We conclude that 
\[ \C \1 = Z(\pi(\g)'') = Z(\pi_1(\g)'') \otimes \1,\] 
so that $\pi_1$ is a factor representation. 
We also see with Schur's Lemma that $\pi$ is irreducible if and only if 
$\pi_1$ is irreducible. 
\end{prf}

\begin{lem} \mlabel{prop:c.23b} Suppose that $\fk$ is compact semisimple. 
Let $(\pi, \cH)$ be a bounded unitary factor (irreducible) representation 
of $\Gamma_c(\fK)$ and $U \subeq X$ an open relatively compact 
subset for which the bundle $\fK$ is trivial on an open neighborhood $V$ 
of $\oline U$. Then the following assertions hold: 
\begin{description}
\item[\rm(i)] The restriction of $\pi$ to the ideal 
$\fn := \Gamma_c(\fK\res_U) \cong C^\infty_c(U,\fk)$ extends to a representation  
of the topological Lie algebra $C^\infty_c(U,\fk) \rtimes \fk 
\cong \fk \otimes_\R C^\infty_c(U)_+$. 
\item[\rm(ii)] There exists a 
bounded factor (irreducible) representation 
$(\pi_1, \cH_1)$ of $\Gamma_c(\fK)$ and a finite tensor product  
$(\pi_2, \cH_2)$ of irreducible evaluation representations
at different points
such that $\pi \cong \pi_1 \otimes \pi_2$ and 
$\fn \subeq \ker \pi_1$. 
\end{description}
\end{lem}

\begin{prf} Let $\cM := \pi(\fn)'' \subeq B(\cH)$ denote the 
bicommutant of $\pi(\fn)$. In view of Lemma~\ref{lem:facrep}, 
this is a factor. If $\pi(\fn) = \{0\}$, there is nothing to show. 
We may therefore assume that $\fn \not\subeq \ker \pi$. 
Since $\fn \trile \fg := \Gamma_c(\fK)$ is an ideal, 
we obtain a homomorphism 
\[ \alpha \: \g \to \der_*(\cM), \quad 
\alpha(s)(A) := [\pi(s),A].\] 
Since every derivation of a von Neumann algebra is inner 
(\cite[Thm.~XI.3.5]{Ta03}), 
$\der(\cM) \cong \cM/Z(\cM)$ as Banach--Lie algebras. 
This in turn implies that the $*$-derivations of $\cM$ are induced 
by elements in $\fu(\cM) := \{ A \in \cM \: A^* = - A\}$, i.e., 
\[ \der_*(\cM) := \{ D \in \der(\cM) \: (\forall M\in \cM)\, 
D(M^*) = D(M)^*\} \cong \fu(\cM)/Z(\fu(\cM)).\]  

Let $h \in C^\infty_c(V,\R)$ be such that 
$h\res_U = 1$. For $x \in \fk$ and the corresponding section 
\[ x \otimes h \in \fk \otimes C_c^\infty(V,\R)  \cong 
C_c^\infty(V,\fk) \cong \Gamma_c(\fK\res_V) \subeq \Gamma_c(\fK) \] 
we then have 
\begin{equation}\mlabel{eq:comder}
\alpha(x \otimes h) \pi(s) = \pi([x, s]), \quad 
s \in C_c(U,\fk).
\end{equation}
In particular, $\alpha(x \otimes h)$ does not depend on the choice 
of $h$, which leads to a homomorphism 
\[ \oline\alpha \: \fk \to \der_*(\cM), \quad 
x \mapsto \alpha(x \otimes h).\] 
The pullback along $\oline\alpha$ of the central extension
\[
\fu(Z(\cM)) \rightarrow \fu(\cM) \rightarrow \der_*(\cM) 
\]
yields a $\fu(Z(\cM))$-valued 2-cocycle on $\fk$, which is trivial 
because $\fk$ is semisimple. We conclude that there exists a 
homomorphism
\[ \tilde\alpha \: \fk \to \fu(\cM) \quad \mbox{ with } \quad 
\alpha(x) = \ad(\tilde\alpha(x)) \quad \mbox{ for } \quad x \in \fk.\] 
Then $\tilde\alpha$ is a bounded unitary representation 
of $\fk$. Because of (\ref{eq:comder}), 
\[ \oline\pi \: \fn_+ := 
C_c(U,\fk) \rtimes \fk \to \fu(\cM) \subeq \fu(\cH), \quad 
(s, x) \mapsto \pi(s) + \tilde\alpha(x) \] 
defines a bounded unitary representation whose range lies in $\fu(\cM)$. 

Let $\cA := C^*(\oline\pi(\fn_+)) \subeq B(\cH)$ denote the 
$C^*$-algebra generated by $\oline\pi(\fn_+)$. 
Then every irreducible representation 
$(\alpha, \cF)$ of $\cA$ defines a bounded irreducible representation of $\fn_+$,
hence is finite-dimensional by Corollary~\ref{cor:tenspro}. 
In particular, the image of $\cA$ in 
every irreducible representation contains the compact operators, 
so that $\cA$ is type~I (\cite{Sa67}). We conclude that the factor 
$\cM = \cA''$ is of type I and that 
$\oline\pi$ is a factor representation of type~I, 
hence a multiple of some 
irreducible representation $(\oline\rho,V)$ whose restriction 
$(\rho,V)$ to $\fn$ is also irreducible. 
Now $\rho$ is a finite tensor product of irreducible evaluation representations 
in different points for
$\fk \otimes C^\infty_c(U,\R)_+$ (cf.\ Theorem~\ref{thm:tenspro}), 
and since the restriction to $\fn$ 
is  irreducible, none of the corresponding characters of 
$C^\infty_c(U,\R)_+$ vanishes on the ideal $C^\infty_c(U,\R)$. 
Hence they are given by evaluations in points of $U$ (Example~\ref{ex:cia}). 
In view of Lemma~\ref{lem:c.21}, $\rho$ extends uniquely to an irreducible 
bounded unitary representation $(\pi_2,\cH_2)$ of $\g$. 
Now Lemma~\ref{lem:facrep}(ii) applies and the assertion follows. 
\end{prf}

\begin{thm} \mlabel{thm:d.25} Suppose that $X$ is a smooth 
manifold with compact boundary and that 
$\fK \to X$ is a smooth Lie algebra bundle whose typical fiber 
$\fk$ is a compact semisimple Lie algebra. 
Let $C \subeq X$ be a compact subset 
and let $(\pi, \cH)$ be a bounded unitary factor (irreducible) representation 
of $\Gamma_c(\fK)$. Then there exists a 
bounded factor (irreducible) representation 
$(\pi_1, \cH_1)$ and a finite tensor product  
$(\pi_2, \cH_2)$ of irreducible evaluation representations in different points 
for $\Gamma_c(\fK)$ 
such that $\pi \cong \pi_1 \otimes \pi_2$ and 
$\Gamma(\fK)_C =  \{ s \in \Gamma(\fK) \: \supp(s) \subeq C\} \subeq \ker \pi_1$.
\end{thm}

\begin{prf} Let $(U_j)_{j \in J}$ be as in Remark~\ref{rem:decomp}. 
Since $C$ is compact and the covering $(U_j)$ is locally finite, 
the set 
$F := \{ j \in J \: U_j \cap C \not=\eset\}$ is finite 
and we may w.l.o.g.\ assume that $F = \{1,\ldots, N\}$. 
Then $\fn := \sum_{j \in F} \fg_j$ is an ideal of $\g := \Gamma_c(\fK)$, 
where $\g_j := \Gamma_c(\fK\res_{U_j}) \cong C_c^\infty(U_j,\fk)$. 
From $C \subeq \bigcup_{j \in F} U_j$ it follows that 
$\supp(s) \subeq C$ implies $s \in \fn$. 

If $\fn \subeq \ker \pi$, we put $\pi_1 := \pi$, and there is nothing 
to show. If this is not the case, 
there exists a minimal $j \in F$ for which $\g_j \not\subeq \ker \pi$. 
Then Lemma~\ref{prop:c.23b} leads to a tensor product 
decomposition 
$\pi = \rho_j \otimes \rho_j'$, where 
$\rho_j$ vanishes on $\sum_{i \leq j} \g_j$ and 
$\rho_j'\res_{\g_j}$ is a finite tensor product of irreducible evaluation 
representations in different points, hence in particular irreducible. 
If $\fn \not\subeq \ker \rho_j$, we apply the same argument to $\rho_j$, 
where $j' \in \{ i \in F \: i > j\}$ is now minimal with $\g_{j'} 
\not\subeq \ker \rho_j$. 
After at most $N$ steps, we arrive at a factorization 
$\pi = \pi_1 \otimes \pi_2,$ 
where $\pi_1$ vanishes on $\fn$ and 
$\pi_2$ is a finite tensor product of irreducible evaluation 
representations. 
Because $\rho_j$ vanishes on $\sum_{i \leq j} \g_j$, the new points
that one obtains at each step cannot coincide with points that one 
already had.
This completes the proof.  
\end{prf}

\subsection{Classification of irreducible bounded representations}

Using the preceding theorem on the representations of 
$\Gamma_c(\fK)$, we can now prove our main result on the 
Fr\'echet--Lie algebra $\Gamma(\fK)$ of all smooth sections. 
In particular, it shows that all irreducible bounded 
unitary representations are finite-dimensional.

\begin{thm} \mlabel{thm:d.24} Suppose that $X$ is a smooth 
manifold with compact boundary and that 
$\fK \to X$ is a smooth Lie algebra bundle whose typical fiber 
$\fk$ is a compact semisimple Lie algebra. 
Then every bounded irreducible unitary representation $\pi$
of $\Gamma(\fK)$ is equivalent to a finite tensor product of irreducible 
evaluation representations at different points.
That is, there exists a finite subset $\bx \subseteq X$
and irreducible representations $\rho_{x}$ of $\fK_{x}$
such that $\pi \simeq \pi_{\bx,\rho} := \bigotimes_{x\in \bx} \rho_{x} \circ \ev_x$. 
Two such representations $\pi_{\bx,\rho}$ and $\pi_{\bx',\rho'}$ are
equivalent if and only if $\bx = \bx'$ and $\rho_{x}\simeq \rho'_{x}$
for all $x\in \bx$.
\end{thm}

\begin{prf} In view of 
Proposition~\ref{prop:6.1b}, we may w.l.o.g.\ assume that 
$X$ is compact. Then $\pi \simeq \pi_{\bx,\rho}$ follows from the 
preceding theorem with $C = X$. 
Since the evaluation map $\ev_{\bx\cup \bx'} : \Gamma(\fK) \rightarrow 
\bigoplus_{x\in \bx\cup \bx'}\fK_x$ is surjective, 
$\pi_{\bx,\rho} \simeq \pi_{\bx',\rho'}$ implies $\bx = \bx'$,
as well as $\rho_{x}\simeq \rho_{x}'$.
\end{prf}

We now consider the Lie algebra $\g = \Gamma_c(\fK)$ for a 
non-compact manifold $X$. 
For every 
bounded factor representation $(\pi, \cH)$ of 
$\g$ and every compact equidimensional submanifold $Y \subeq X$ with boundary, 
we have seen 
in Theorem~\ref{thm:d.25} that
there exists a factorization 
$\pi = \pi_1 \otimes \pi_2$ for which 
$\Gamma(\fK)_Y \subeq \ker \pi_1$, 
and $\pi_2$ is a finite tensor product of evaluation representations 
$\bigotimes_{x\in \bx} \rho_{x} \circ \ev_{x}$, where 
$(\rho_x, V_x)$ are irreducible representations of the Lie algebras 
$\fK_{x} \cong \fk$. Therefore 
\[ C^*(\pi(\Gamma(\fK)_Y)) \cong C^*(\pi_2(\Gamma(\fK)_Y)).\] 
We may then assume w.l.o.g.\ that $\bx \subseteq Y^0$, which 
further implies that 
$\pi_2(\Gamma(\fK)_Y) \cong \oplus_{x\in \bx} \rho_x (\fK_{x})$, and thus 
\[ C^*(\pi_2(\Gamma(\fK)_Y)) \cong \bigotimes_{x\in \bx} B(V_x).\] 

Since $X$ is $\sigma$-compact, 
$X = \bigcup_n Y_n$ with $Y_n \subeq Y_{n+1}^0$ and $Y_n$ is a compact submanifold 
with boundary, so we can iterate the preceding construction. Therefore 
$\Gamma_c(\fK) = \bigcup_n \Gamma(\fK)_{Y_n}$ implies the existence of a 
locally finite subset $\bx \subseteq X$ such that 
\[ C^*(\pi(\Gamma_c(\fK))) 
\cong \indlim C^*(\pi(\Gamma(\fK)_{Y_n})) 
\cong \indlim \bigotimes_{x\in \bx \cap Y_n^0} B(V_x)
=: \hat\bigotimes_{x\in \bx} B(V_x)
\,,\] 
where the second limit is 
the direct limit of the net of $C^*$-algebras $\bigotimes_{x\in \bx \cap Y_n^0}B(V_x)$ 
over the directed system of finite subsets $\bx \cap Y_n^0$ of $\bx$,
i.e.,
the norm completion of the algebraic limit.
Defining
\[
\cA_{\bx,\rho} := \hat\bigotimes_{x\in \bx} B(V_x)\,,
\]
we obtain for every $x$ a canonical inclusion $\iota_x \: B(V_x) \hookrightarrow \cA_{\bx,\rho}$.

Conversely, for every locally finite subset 
$\bx \subeq X$ with corresponding irreducible unitary representations $\rho_x$ 
of $\fK_{x}$, we obtain a 
Lie algebra homomorphism 
\begin{equation}
  \label{eq:iota}
\eta_{\bx,\rho} \: \Gamma_{c}(\fK) \to \cA_{\bx,\rho}, 
\quad s \mapsto 
\sum_{x\in \bx}
\iota_{x} \circ \rho_{x}(s(x))\,
\end{equation}
whose image generates a dense subalgebra. 
(The sum only has finitely many terms.)

The preceding discussion now leads to the following theorem which describes  
the bounded irreducible and factor representations of the 
LF-Lie algebra $\Gamma_c(\fK)$ in terms of irreducible representations of the 
$C^*$-algebras $\cA_{\bx,\rho}$. It reduces all Lie theoretic issues 
concerning these representation to questions concerning $C^*$-algebras. 

\begin{thm} \mlabel{thm:4.12} 
Suppose that $X$ is a smooth 
manifold with compact boundary and that 
$\fK \to X$ is a smooth Lie algebra bundle whose typical fiber 
$\fk$ is a compact semisimple Lie algebra. 
For every bounded unitary  factor (irreducible) representation 
$(\pi, \cH)$ of $\Gamma_c(\fK)$, there exists a locally finite subset 
$\bx \subeq X$, irreducible representations $\rho_{x}$ of $\fK_{x}$ corresponding to 
$x\in \bx$,
and a unique factor (irreducible) representation 
$\beta \: \cA_{\bx,\rho} \to B(\cH)$ with 
$\pi \simeq \beta \circ \eta_{\bx,\rho}$. 
Conversely, every representation of the type  
$(\beta \circ \eta_{\bx,\rho},\cH)$
is a bounded factor (irreducible) representation.
Two such representations are equivalent,
$\beta \circ \eta_{\bx,\rho} \simeq \beta' \circ \eta_{\bx',\rho'}$,
if and only if $\bx = \bx'$, $\rho_x \simeq \rho'_x$ for all 
$x\in \bx$, and $\beta \simeq \beta'$.
\end{thm}

\begin{rem}
The algebra $\cA_{\bx,\rho}$ is a so-called UHF (ultra hyperfinite) $C^*$-algebra. 
These algebras have been classified by 
Glimm in \cite{Gli60}, where one also finds a characterization 
of their pure states. Even stronger results were obtained later by 
Powers in \cite{Po67}, where he  shows that the automorphism group 
acts transitively on the set of pure states, so that every irreducible 
representation is a twist (by an automorphism) of an infinite 
tensor product of irreducible representations. 
Together with these results, Theorem~\ref{thm:4.12} 
provides a complete description of the 
bounded irreducible unitary representations of $\Gamma_c(\fK)$.  
\end{rem}

\begin{ex} \mlabel{ex:a.5}
If $\cA = \hat\bigotimes_{x\in\bx} B(V_x)$ is an UHF $C^*$-algebra, 
then typical examples of irreducible representations 
are the infinite tensor products 
$\hat\otimes_{x\in\bx} (V_x, v_x)$, where 
$v_x \in V_x$ is a unit vector. All these representations 
are irreducible (\cite[Prop.~5.2.1]{Sam91}; 
see also \cite[Prop.~4.4.3]{Sa71}), but they do not exhaust all 
irreducible representations. The restriction of these representations 
to the canonical maximal abelian subalgebras is multiplicity free. 
This is not true for all irreducible unitary representations 
 (cf.~\cite[Sect.~5.2]{Sam91}). Two such infinite tensor product representations 
corresponding to the sequences $(v_x)$ and $(w_x)$ of unit vectors 
are unitarily equivalent if and only if 
\[ \sum_{x\in \bx} 1 - |\la v_x, w_x \ra| < \infty \] 
(\cite[Prop.5.2.2]{Sam91}). 
In particular, there exist infinitely many non-equivalent irreducible unitary representations. 
The above condition is equivalent to 
\[ \sum_{x\in \bx}^\infty d([v_x], [w_x])^2 < \infty\,, \] 
since the natural metric on the projective 
space $\bP(V_{x})$ satisfies  
$d([v], [w])^2 = 2(1 - |\la v,w \ra|)$ 
(\cite[Lemma~3.2]{Ne12}).

\end{ex}

\subsection{Translation to the group context}
We have carried out our classification of the bounded 
unitary representations on the Lie algebra level, which is equivalent to 
working with the corresponding $1$-connected groups. However, 
some natural mapping groups, such as $G := C^\infty(X,K)$, where $X$ is compact 
(take f.i.~$X = \bS^3$) and 
$K$ is a $1$-connected compact group, are neither connected nor 
simply connected. 

Let $P \rightarrow X$ be a principal $K$-bundle over a {\em compact} space $X$,
with $K$ compact semisimple.
Since every irreducible bounded representation of 
$\Gamma(\Ad(P))$ is a finite tensor product of irreducible evaluation representations 
and every irreducible evaluation representation obviously integrates to a representation 
of $\Gau(P)$, all bounded unitary representations do. They actually 
factor through quotient homomorphisms $\Gau(P) \to \Pi_{x\in \bx} \Gau(P|_{x})$, given by evaluation
in the finite subset $\bx \subseteq X$.
Each factor $\Gau(P|_{x})$ is isomorphic to $K$, so 
$\Pi_{x\in \bx} \Gau(P|_{x})$ is isomorphic to the
$1$-connected group $K^{\bx}$. 
In particular, the group $\pi_0(\Gau(P))$ acts trivially 
on the set 
$\hat{\Gau}(P)_0{}^b$ 
of equivalence classes of bounded irreducible 
unitary representations of the identity component $\Gau(P)_0$. 
If $\Gau(P)$ is not connected, then a bounded irreducible 
unitary representation is not determined by its derived 
Lie algebra representation. 
In this context however, we do have the following theorem:

\begin{thm}
Every irreducible bounded unitary representation of $\Gau(P)$
is equivalent to $\pi_1 \otimes \pi_2$, where 
$\pi_1$ is a finite tensor product of irreducible evaluation representations
at different points,
and $\pi_2$ comes from an irreducible unitary 
representation of the discrete group $\pi_0(\Gau(P))$. 
Conversely, any such tensor product is irreducible.
\end{thm}
\begin{prf}
We set $G := \Gau(P)$.
Let $\cA := C^*(\pi(G_0)) = C^*(\dd\pi(\g)) \subeq B(\cH)$ denote the 
$C^*$-algebra generated by $\pi(G_0)$. Then every irreducible 
representation $\beta$ of $\cA$ defines a bounded unitary representation of $G_0$,
hence is finite-dimensional. In particular, the image of $\cA$ in 
every irreducible representation contains the compact operators, 
so that $\cA$ is type~I (\cite{Sa67}). 
The group $G$ acts by conjugation on $\cA$ 
but we have already seen above that its action on the space 
$\hat\cA$ of equivalence classes of irreducible representations of 
$\cA$ is trivial. 

Since $\cH$ and $\cA$ are separable and $\cA$ is of type~I, 
the representation of $\cA$ on $\cH$ has a canonical direct integral decomposition 
\[ \cH 
\cong \int^\oplus_{\hat\cA} \cH_\alpha\, d\mu(\alpha)
\cong \bigoplus_{n \in \N \cup \{\infty\}} \int_{\hat\cA_n} \cH_\alpha\, d\mu_n(\alpha) \] 
for a measure $\mu$ on $\hat\cA$ which is a sum of disjoint measures 
$\mu_n := \mu\res_{\hat\cA_n}$, $n \in \N \cup \{\infty\}$, where 
$(\hat\cA_n)_{n \in \N \cup \{\infty\}}$ is a measurable partition of $\hat\cA$, 
and for which the representation on $\cH_\alpha$, $\alpha \in \hat\cA_n$, 
is an $n$-fold multiple of an irreducible representation of type $\alpha$ 
(\cite[Thm.~8.6.6]{Dix77}). 
Here the measure classes $[\mu_n]$, $n \in \N \cup \{\infty\}$, 
are uniquely determined by the representation. 
Since $G$ acts trivially on $\hat\cA$, it preserves all these measure 
classes. Therefore the irreducibility of the representation 
implies that only one of these measures is non-zero, and that this measure 
is ergodic for the $G$-action on $\hat\cA$, hence a point measure 
because the action is trivial. 
Therefore $\pi\res_{G_0}$ is a factor representation of type I, 
hence a multiple of a tensor product 
of irreducible evaluation representations at different points. This implies that 
$\pi \cong \pi_1 \otimes \pi_2$, where 
$\pi_1$ is a finite tensor product of irreducible evaluation representations
at different points,
and $\pi_2$ vanishes on $G_0$, hence defines an irreducible unitary 
representation of the discrete group $\pi_0(G)$. Conversely, any such tensor product 
is irreducible. 
\end{prf}

\section{Noncompact fibers and projective representations} 
\mlabel{sec:5}

In this section we show that the problem of classifying bounded 
irreducible \emph{projective} unitary representations 
$\overline{\pi} \: \Gamma_{c}(\fK) \rightarrow \fp\fu (\cH)$,
where the typical fiber $\fk$ of $\fK$ is an \emph{arbitrary}
finite-dimensional real Lie algebra, reduces to Theorems \ref{thm:d.24} and \ref{thm:4.12}.
This justifies our assumptions that $\fk$ is compact semisimple
and that $\pi \: \Gamma_{c}(\fK) \rightarrow \fu (\cH)$ is a linear 
representation.


We start with the observation that a finite-dimensional Lie algebra 
$\fk$ is compact if and only if its adjoint group 
$\la e^{\ad \fk} \ra$ is relatively compact in $\GL(\fk)$, which in turn is equivalent 
to the existence of an adjoint invariant norm on $\fk$. 
Next we observe that every finite-dimensional Lie algebra $\fk$ 
contains a unique minimal ideal $\fn \trile \fk$ for which 
$\fk/\fn$ is compact. This is a direct consequence of the fact that 
direct sums and subalgebras of compact Lie algebras are compact, 
which implies that the set of all ideals with compact quotients is filtered.

\begin{prop} \mlabel{prop:6.3} Let $\cA$ be a real commutative associative locally convex 
algebra and let $p$ be a bounded seminorm on $\fk \otimes \cA$ 
which is invariant under $e^{\ad \fk}$. Let 
$\fn \trile \fk$ be the minimal ideal for which 
$\fk/\fn$ is a compact Lie algebra. Then $\fn \otimes \cA \subeq p^{-1}(0)$. 
\end{prop}

\begin{prf} Let $\g = \fk \otimes \cA$. 
For each $a \in \cA$ we consider the seminorm 
$p_a(x) := p(x \otimes a)$ on $\fk$. Since the 
map $\fk \to \fk \otimes \cA, x \mapsto x \otimes a$ is 
$\fk$-equivariant, the seminorm $p_a$ on $\fk$ is invariant. 
Therefore $p_a$ induces on the quotient 
$\fk/p_a^{-1}(0)$ an invariant norm, so that this Lie 
algebra is compact. This implies that 
$\fn \subeq p_a^{-1}(0)$, and hence that 
$\fn \otimes \cA \subeq p^{-1}(0)$. 
\end{prf}

\begin{cor} \mlabel{cor:1} Let $p$ be a bounded seminorm on $C^\infty(X, \fk)$ or 
$C^\infty_c(X,\fk)$, which is invariant under the adjoint action and 
let $\fn \trile \fk$ be the minimal ideal for which 
$\fk/\fn$ is a compact Lie algebra. Then 
$C^\infty(X,\fn)$ and 
$C^\infty_c(X,\fn)$, respectively, are contained in the closed ideal $p^{-1}(0)$. 
\end{cor}

\begin{prf} The Lie algebra $C^\infty(X,\fk)$ is covered by the 
preceding proposition because $C^\infty(X,\R)$ is unital, 
so that $\fk \cong \fk \otimes \1 \subeq \fk \otimes \cA$. 

For $\cA = C^\infty_c(X,\R)$, we fix a compact subset 
$Y \subeq X$ and consider the subalgebra 
\[ \cA_Y := C^\infty(X,\R)_Y := \{ f \in C^\infty(X,\R) \: \supp(f) \subeq Y\}.\] 
Let $\chi \in C^\infty_c(X,\R)$ with $\chi\res_Y = 1$. 
For $x \in \fk$ we then have 
\[ \ad (x \otimes \chi)(y \otimes a) = [x,y] \otimes a 
\quad \mbox{ for } \quad a \in \cA_Y.\] 
Therefore the restriction of $p$ to 
$\fk \otimes \cA_Y$ is invariant under 
$e^{\ad \fk}$, so that the assertion follows from Proposition~\ref{prop:6.3}.
\end{prf}

\begin{cor} Let $\fK \to X$ be a $\fk$-Lie algebra bundle, where 
$\fk$ is a finite-dimensional real Lie algebra, 
let $\fn \trile \fk$ be the minimal ideal for which 
$\fk/\fn$ is a compact Lie algebra, and let
$\fN \subeq \fK$ be the corresponding subbundle. 
If $\oline\pi \: \Gamma_c(\fK) \to \pu(\cH)$ 
is a projective unitary representation, 
then $\Gamma_c(\fN) \subeq \ker \oline\pi$. 
\end{cor}

\begin{prf} Since the operator norm on $\fu(\cH)$ is invariant under the 
adjoint action of $\U(\cH)$, it induces on 
the Banach--Lie algebra $\pu(\cH) = \fu(\cH)/i\R \1$ 
a norm which is also invariant under the adjoint action. 
Therefore $p(s) := \|\oline\pi(s)\|$ is an invariant seminorm 
on $\Gamma_c(\fK)$. We have to show  that 
$\Gamma_c(\fN) \subeq p^{-1}(0)$. 

In view of Remark~\ref{rem:decomp}, $\Gamma_c(\fK)$ 
is a sum of ideals of the form 
$\Gamma_c(\fK\res_U) \cong C^\infty_c(U,\fk)$, where 
$\fK\res_U$ is trivial. It therefore suffices to observe that 
Corollary~\ref{cor:1} implies  
$C^\infty_c(U,\fn) \cong \Gamma_c(\fN\res_U) \subeq p^{-1}(0)$. 
\end{prf}

The preceding corollary shows that, for the sake of classifying bounded projective 
unitary representations of the Lie algebra $\Gamma_c(\fK)$, we may w.l.o.g.\ 
assume that $\fk$ is compact.
As the following theorem shows, 
the corresponding cocycles must be trivial.

\begin{thm} 
If $\fK$ is a bundle of Lie algebras with compact fiber, 
then every bounded projective unitary representation 
of $\Gamma_{c}(\fK)$ lifts to a bounded unitary representation.
\end{thm} 

\begin{prf} Let $\pi \colon \Gamma_{c}(\mathfrak{K}) \rightarrow \mathfrak{u}(\mathcal{H})$
be a bounded projective representation. Then there exists a cocycle
$\omega \colon \Gamma_{c}(\mathfrak{K}) \times \Gamma_{c}(\mathfrak{K}) \rightarrow
\mathbb{R}$ such that $\pi$ is a unitary representation of 
$\R c \oplus_\omega \Gamma_{c}(\fK)$ mapping $c$ to $i\mathbf{1}$.
As the decomposition $\fk = \mathfrak{z}(\fk) \oplus [\fk,\fk]$ is $\mathrm{Aut}(\fk)$-invariant, 
it leads to a corresponding direct sum $\fK = Z(\fK) \oplus [\fK,\fK]$ of Lie algebra bundles, 
and hence to $\Gamma_{c}(\fK) = \Gamma_{c}(\mathfrak{z}(\fK)) \oplus \Gamma_{c}([\fK,\fK])$.

First of all, we show that $\omega$ vanishes on 
$\Gamma_{c}(\mathfrak{z}(\fK)) \times \Gamma_{c}(\mathfrak{z}(\fK))$.
Indeed, for all 
$z,z' \in \Gamma_{c}(\mathfrak{z}(\fK))$,
$\pi$ defines a bounded representation of the at most $2$-step nilpotent Lie algebra 
spanned by $z$, $z'$ and $c$, and thus vanishes on 
$[z,z'] = \omega(z,z')c$. 
This implies that $\omega(z,z') = 0$. 
Furthermore, the cocycle property implies that $\omega(z,[\xi,\eta]) = 0$ for all 
$z\in \Gamma_{c}(\mathfrak{z}(\fK))$
and $\xi,\eta \in \Gamma_{c}(\fK)$, so that we may assume w.l.o.g.\ 
that $\fk$ is compact semisimple.
 
Let $\mathbb{V} := S^2(\fK)/\langle \mathrm{ad}(\fK)\cdot S^2(\fK) \rangle$, $\kappa \colon 
\fK \times \fK \rightarrow \mathbb{V}$ be the universal invariant symmetric 
bilinear bundle map. Then $\bV$ is a flat bundle, hence carries a canonical 
differential $\dd \: \Gamma(\bV) \to \Omega^1(X,\bV)$.  
We know from \cite{JW10} that 
there exists a Lie connection $\nabla$ on $\fK$ and 
a continuous linear functional $\lambda$ 
on $\Omega_{c}^1(M,\bV)$ vanishing on $\mathbf{d}\Gamma_c(\bV)$, such that 
$\omega$ is cohomologous to the cocycle 
$(\xi,\eta) \mapsto \lambda(\kappa(\xi,\nabla\eta))$. 
Suppose that $\lambda \neq 0$. Then there exists a $\gamma \in \Omega_{c}^1(M,\mathbb{V})$
with $\lambda(\gamma) \neq 0$. We may assume w.l.o.g.\ that $\gamma$
is supported in an open set $U \subset M$ over which $\fK$ (and thus $\mathbb{V}$)
can be trivialized.
Using this trivialization, 
we write $\gamma = \sum_{j} f_j \mathbf{d}g_j \otimes \kappa (X_j,Y_j)$.
By polarization, we may assume that $X_j = Y_j$, so that we find 
$f\otimes X$ and $g\otimes X$ with $\lambda(\kappa (f\otimes X, \nabla g \otimes X)) 
\neq 0$.
As $\pi$ defines a bounded unitary representation of the $2$-step nilpotent 
Lie algebra spanned by $f\otimes X$, $g\otimes X$ and $c$, we must have  
$\lambda(\kappa (f\otimes X, \nabla g \otimes X))~=~0$, contradicting 
our hypothesis. 
This means that $\omega$ is a coboundary, and $\pi$ lifts to a 
bounded unitary representation of $\Gamma_{c}(\fK)$. 
\end{prf}

By Remark \ref{rem:4.2}, we may assume w.l.o.g.\ that  
$\fk$ is compact semisimple, which renders Theorems \ref{thm:d.24} and \ref{thm:4.12}
applicable.

\section{Boundary conditions and non-unital algebras}  
\mlabel{sec:6}

In this final section we discuss 
Lie algebras of the form 
$\fk_\cA= \fk \otimes_\R \cA_\R$ for $\cA$ non-unital,
and Lie algebras of sections of $\fK$ that satisfy vanishing conditions  
at the boundary of $X$. 

First we show that, for $\cA = \ell^1(\N,\C)$  
(with the pointwise product), the 
Lie algebra $\fk_\cA$ has infinite-dimensional bounded irreducible representations 
(Subsection~\ref{subsec:6.1}). 
We aready know from Section~\ref{sec:4} that this is the case for 
$\cA = C^\infty_c(X,\R)$, where $X$ is a non-compact manifold, but the case 
of $\ell^1(\N,\C)$ shows that this also happens for Banach algebras. 
In view of this observation, it is remarkable that this 
phenomenon does not occur for non-unital $C^*$-algebras, as we show in 
Subsection~\ref{subsec:6.2}. 
Given these two classes of examples, one expects that, for a 
Banach algebra $\cA$ with $\ell^1(\N,\C) \subeq \cA \subeq c_0(\N,\C)$, 
a mixture of the ``tame'' behavior of $c_0(\N,\C)$, where all 
irreducible representations are finite-dimensional, and the ``wild'' behavior 
for $\ell^1(\N,\C)$ will occur. 

This issue is addressed on a quantitative level in Subsection~\ref{subsec:6.3}, 
where we discuss the Banach--Lie algebra 
$\Gamma^k_0(\fK)$ of $C^k$-sections of $\fK$ 
whose $k$-jet vanishes at the boundary $\partial X$ of $X$. This is a 
Banach completion of $\Gamma_c(X^\circ)$ and we characterize those 
pairs $(\bx, \rho)$ for which the homomorphism 
$\eta_{\bx,\rho} \:  \Gamma_c(\fK) \to \cA_{\bx,\rho}$ extends continuously 
to $\Gamma^k_0(\fK)$. 

\subsection{Bounded representations of \texorpdfstring{$\ell^1(\N, \fk)$}{l1(N,k)}} 
\mlabel{subsec:6.1}

The simplest non-compact manifold is $X = \N$. In this case 
\[ C^\infty_c(X,\fk) \cong \fk^{(\N)}
:= \indlim \fk^N := \{ X = (X_n) \in \fk^\N \: 
|\{n \in \N \: X_n \not=0\}| < \infty\}.\] 
From Theorem~\ref{thm:4.12} it follows that $\fk^{(\N)}$ 
has a wild bounded unitary representation theory. 
That this is not a phenomenon caused by the rather fine topology 
on this Lie algebra, follows from the fact that it is shared by 
certain Banach completions. 
Let $\rho := (\rho_n, V_n)_{n\in\N}$ be a sequence of irreducible unitary 
representations of $\fk$ and let
$\cA_\rho := \hat\otimes_{n \in \N} B(V_n)$  denote the 
corresponding UHF $C^*$-algebra. We further assume that 
$\sup \|\rho_n\| < \infty$. Then the inclusion
 $\fk^{(\N)} \to \cA_\rho$ extends to a continuous embedding 
\begin{equation}
  \label{eq:inftens}
 \eta_\rho \: \g := \ell^1(\N,\fk) 
\to \cA_\rho, \quad 
(X_n)_{n \in \N} \mapsto 
\sum_{n = 1}^\infty \1^{\otimes (n -1)} \otimes \rho_n(X_n) \otimes \1^{\otimes\infty} 
\end{equation}
with $\|\eta_\rho\| \leq  \sup_{n \in \N} \|\rho_n\|.$ 
Since $\eta_\rho$ maps $\fk^{(\N)}$ to a topologically generating 
subalgebra,  $\im(\eta_\rho)$ generates a dense subalgebra. 
Therefore the Banach--Lie algebra 
$\g$ has bounded factor representations of type II and III. 

\begin{ex} \mlabel{ex:6.7} (cf.\ \cite[pp.~205~et~seqq.]{Sa71}) 
For $\fk = \su_2(\C)$ and the defining representation  
$\rho_n \: \fk \to M_2(\C)$, we consider the algebra $\cA := \cA_\rho$, i.e., 
$\cA = \hat\otimes_{n \in \N} M_2(\C)$. 
Let $0 \leq p_n \leq \shalf$ 
and consider the factorial state 
\[ \phi_n\pmat{a & b \\ c & d} := p_n a + (1-p_n) d \] 
on $M_2(\C)$. Then $\psi := \otimes_{n \in \N} \phi_n$ is a factorial 
state on $\cA$. If $p_n = 0$ for every $n$, then $\psi$ is pure, so that 
we obtain a type I$_\infty$ representation. If 
$p_n = \shalf$ for every $n$, then 
$\pi_\psi(\cA)''$ is a type II$_1$-factor. If there exists a 
$\delta > 0$ with 
$\delta < p_n < \shalf - \delta$ for every $n$, then 
$\pi_\psi(\cA)''$ is a type III factor. 
For $p_n = \lambda \in ]0,\shalf[$, $n \in \N$, the factor 
$\cM_\lambda := \pi_\psi(\cA)''$ is called a {\it Powers factor} 
(cf. \cite{Po67}). 
\end{ex}

\begin{remark}
It is easy to see that the map $\eta_\rho$ actually is an isometric 
embedding $\ell^1(\N,\fk) \into \cA$. Let $\beta \:  \cA \to B(\cH)$ 
be a factor representation of $\cA$. Since $\cA$ is simple, 
$\beta$ is isometric, and this implies for every sequence 
$(X_n) \in \fk^\N$ that 
\[\lim_{N \to \infty} \|\beta \circ \eta_\rho(X_1,X_2, \ldots, X_N,0,\cdots)\| 
= \sum_n \|X_n\|.\] 
In particular, the representation $\pi := \beta \circ \eta_\rho$ does 
not extend to 
a bounded representation of 
$\ell^2(\N,\fk)$, or any other  
Banach--Lie algebra containing $\ell^1(\N,\fk)$ as a proper dense subspace. 

As the spectra of elements in $\ell^1(\N,\fk)$ are symmetric, it follows 
that the map 
\[ \oline\eta_\rho \: \ell^1(\N,\fk) \to \der(\cA), \quad 
X \mapsto \ad(\eta_\rho(X)) \] 
is isometric. Therefore, even the bounded \emph{projective} 
representation $\oline\pi \: \ell^1(\N,\fk) \to \pu(\cH)$ 
does not extend to any Banach--Lie algebra containing 
$\ell^1(\N,\fk)$ as a proper dense subspace. 

However, one can show that for certain $\beta$, $\pi$ has a unique proper extension 
to a projective unitary representation of $\ell^2(\N,\fk)$ by unbounded 
operators and that this representation integrates to an analytic representation 
of a non-trivial central $\T$-extension of the corresponding Banach--Lie group 
\[ \ell^2(\N,\SU_2(\C)) :=\Big \{ k \in \SU_2(\C)^\N \: 
\sum_{n = 1}^\infty \|\1 - k_n\|^2 < \infty \Big\}. \] 
We shall explore this and related phenomena in subsequent work 
(cf.\ \cite{JN13}). 
\end{remark}

\begin{ex}\mlabel{ex:6.3}
Note that any Lie algebra that has a dense continuous homomorphism
into $\ell^1(\N,\fk)$ will inherit the ``wild''
factor representations mentioned in Example~\ref{ex:6.7}.
For example, the Lie algebra
$\cS(\R^d, \fk)$ of $\fk$-valued Schwartz functions
allows for the dense homomorphism 
$\cS(\R^d,\fk) \to \ell^1(\N,\fk)$ defined by  
$f \mapsto (f(ne_1))_{n \in \N}$.
\end{ex}

\begin{ex}\mlabel{ex:6.4}
Consider
the Lie algebra
$\fk_{\cA} := \fk\otimes_\R \cA_\R$, where 
$\cA := C^k_0([0,1],\C)$ is the commutative
Banach algebra
\[ C^k_0([0,1],\C) := \left\{ f \in C^k([0,1],\C)\: f^{(j)}(0) = 0, \ 
\forall \,\, j \in \{0, \ldots, k\}
\right\}\,.\] 
If $k\geq 1$, then we have
a dense continuous Lie algebra homomorphism into $\ell^1(\N,\fk)$, derived from the 
Banach algebra homomorphism
$\eta \: \cA \to \ell^1(\N,\C)$ defined by 
$\eta(f)_n = f(x_n)$, with $x_n := \frac{1}{(n+1)^2}$.
It is continuous 
because 
$|f(x)| \leq \frac{1}{k!} x^k \|f^{(k)}\|_{\infty}$,
and $1/(n+1)^{2k}$ is summable for $k\geq 1$.
Consequently, $\fK_{\cA}$ has  bounded unitary factor representations of type II
and III if $k\geq 1$.
However, we
will see in Section \ref{subsec:6.2} that, for $k=0$, all 
irreducible bounded unitary representations are finite-dimensional,
so that in this case, the bounded unitary factor representations are all of type I.
\end{ex}

In the following, we will perform a more refined analysis of Example \ref{ex:6.4}. 
In Subsection \ref{subsec:6.2} below, we will show that if $\cA := C_0(X,\C)$ is a 
(possibly non-unital) $C^*$-algebra, then every irreducible bounded unitary representation
of $\fk_{\cA}$ is a finite tensor product of evaluation representations.
Note that the Banach algebra $C^k_0([0,1],\fk)$ is a $C^*$-algebra
only if $k=0$. This is what causes the ``tame'' behavior that distinguishes it 
from its siblings with $k\geq 1$.
In Subsection \ref{subsec:6.3} we will take up the thread for $k\geq 1$.
We will show, again in a more general context, that an infinite evaluation 
representation in a sequence $(x_n)$ is defined if and only if $x_n^k$ is summable.

\subsection{Non-unital \texorpdfstring{$C^*$}{C*}-algebras} 
\mlabel{subsec:6.2}

Let $\cA \cong C_0(X)$ be the $C^*$-algebra of continuous functions that vanish at infinity,
for $X$ a locally compact space which is not compact. 
The dual $\cA_\R'$ of the Banach space $\cA_\R \cong C_0(X,\R)$ 
can be identified with the space $\cM(X)$ of finite regular Borel measures on $X$. 
Let $B \subeq \cA_\R'$ be a weak-$*$-compact subset. Then we have 
a natural map $\cA_\R \to C(B,\R), a \mapsto a^*$ with 
$a^*(\alpha) := \alpha(a)$. 

\begin{lem} \mlabel{lem:c.10} 
If $(\delta_n)_{n \in \cA}$ is an approximate identity in 
$\cA$ with $0 \leq \delta_n \leq 1$, then $\delta_n^* \to 1^*$ holds pointwise on $\cA'$. 
\end{lem}

\begin{prf} If $\mu \in \cA' \cong \cM(X)$ is a complex regular measure on 
$X$, then we have to show that 
\[ \int_X \delta_n \, d\mu \to \int_X 1 \, d\mu = \mu(X).\] 
Since $\mu$ is a linear combination of four positive measures 
(\cite[Thm.~6.14]{Ru86}), we may 
w.l.o.g.\ assume that $\mu$ is positive. 
Let $\eps > 0$ and pick a compact subset $K \subeq X$ with 
$\mu(X\setminus K) \leq \eps$. 
There exists a function 
$\chi \in C_0(X)$ with $\chi\res_K = 1$. Then 
$\delta_n \chi \to \chi$ implies that $\delta_n\res_K$ converges uniformly 
to $1$, so that 
$\int_K \delta_n\, d\mu \to \mu(K).$ 
Since 
\[ 0 \leq\int_{X\setminus K} \delta_n\, d\mu \leq 
\mu(X\setminus K) \leq \eps,\] 
it follows that, eventually, 
$\big|\int_X \delta_n\, d\mu - \mu(X)\big| \leq 2 \eps.$ 
\end{prf}

\begin{rem} A sequence $(\delta_n)_{n \in \N}$ with the above properties 
exists if and only if $X$ is countable at infinity. 
In fact, if $X$ is countable at infinity, then there exists an 
exhaustion $(K_n)_{n \in \N}$ by compact subsets satisfying 
$K_n \subeq K_{n+1}^0$. Then Urysohn's Lemma implies the existence of 
$\delta_n \in C_0(X)$ with $\delta_n\res_{K_n} = 1$ and 
$0 \leq \delta_n \leq 1$. Now $(\delta_n)$ is an approximate 
identity of $C_0(X)$. 

If, conversely, $(\delta_n)$ exists, then we consider the 
compact subsets 
$K_n := \big\{ \delta_n \geq \frac{1}{n}\big\} \subeq X.$ 
Since $\delta_n \to 1$ holds uniformly on every compact subset $K \subeq X$, 
there exists an $n \in \N$ with $K \subeq K_n$. In particular, 
$X = \bigcup_n K_n$ and $X$ is countable at infinity. 
\end{rem}

\begin{lem} Let $B \subeq \cA_\R'$ be a weak-$*$-compact subset
with Borel $\sigma$-algebra $\fB(B)$, and 
let $(\delta_n)$ be an approximate identity of $\cA$ with 
$0 \leq \delta_n \leq 1$. For every spectral measure 
$P \: \fB(B) \to B(\cH)$, we then have 
\[ P(1^*) = \slim_{n \to \infty} P(\delta_n^*).\] 
\end{lem}

\begin{prf} Since $B$ is weak-$*$-compact, it is weak-$*$-bounded, 
hence bounded by the Uniform Boundedness Principle. Therefore the 
sequence $(\delta_n^*)$ is uniformly bounded. Since 
$\delta_n^* \to 1^*$ pointwise by Lemma~\ref{lem:c.10}, the assertion follows from the 
standard continuity properties of spectral measures. 
\end{prf}

\begin{prop} \mlabel{prop:c.13} 
Let $\fk$ be a compact Lie algebra and 
$\fk_\cA := C_0(X,\fk)$ for a locally compact space $X$ countable at infinity. 
Then any bounded unitary representation 
$\pi \: \fk_\cA \to \fu(\cH)$ extends to the Banach--Lie algebra 
$\fk_{\cA_+} \cong \fk_\cA \rtimes \fk$ by 
\[ \hat\pi(x) := \slim_{n \to \infty} \pi(x \otimes \delta_n).\] 
\end{prop}

\begin{prf} For each $x \in \fk$, the map 
\[ \pi_x \: \cA_\R \to \fu(\cH), \quad a \mapsto \pi(x \otimes a) \] 
is a bounded representation of the abelian Lie algebra $\cA_\R$. 
Hence there exists a spectral measure 
$P_x$ on a weak-$*$-compact subset $B_x \subeq \cA_\R'$ such that 
\[ \pi_x(a) = i P_x(a) = i \int_{B_x} a \, dP_x \] 
(\cite[Thm.~4.1]{Ne09}). 
We now put 
\[ \hat \pi(x) := i P_x(1^*) = \slim_{n \to \infty} i P_x(\delta_n).\] 
Since the commutator bracket is separately strongly continuous on bounded subsets 
of $B(\cH)$, we obtain 
\begin{align*}
[\hat\pi(x), \pi(y \otimes a)] 
&= \slim_{n \to \infty} [\pi(x \otimes \delta_n), \pi(y \otimes a)] \\
&= \slim_{n \to \infty} \pi([x,y] \otimes \delta_na) 
= \pi([x,y] \otimes a)
\end{align*}
and 
\begin{align*}
[\hat\pi(x), \hat\pi(y)] 
&= \slim_{n \to \infty} [\hat\pi(x), \pi(y \otimes \delta_n)] 
= \slim_{n \to \infty} \pi([x,y] \otimes \delta_n) 
= \hat\pi([x,y]).
\end{align*}
Therefore 
\[ \hat\pi \: \fk_{\cA_+} = \fk_\cA \rtimes \fk \to B(\cH), \quad 
(x \otimes a,y) \mapsto \pi(x \otimes a) + \hat\pi(y) \] 
is a representation of the Banach--Lie algebra $\fk_{\cA_+}$. 
Since $\fk$ is finite-dimensional, it is also continuous. 
\end{prf}

Since every bounded unitary representation of $\fk_{\cA}$ extends to
$\fk_{\cA_+}$,
the classification of the 
irreducible bounded unitary representations of $\fk_{\cA_+}$ 
(Theorem~\ref{thm:6.10}) yields 
immediately the corresponding classification for $\fk_\cA$ 
(cf.\ \cite{NS11} for the case of unital commutative $C^*$-algebras). 

\begin{thm} \mlabel{thm:c.14} 
 Let $\fk_\cA = \fk \otimes_{\R} C_0(X,\C)_{\R}$, with $\fk$ a compact semisimple Lie algebra, and 
 $\cA = C_0(X,\C)$ the commutative $C^*$-algebra of continuous functions 
 vanishing at infinity for a locally compact space $X$ countable at infinity. 
Then every bounded irreducible unitary representation $(\pi, \cH)$ 
of $\fk_\cA$ is a finite tensor product of evaluation representations at different points 
corresponding to irreducible representations of $\fk$. 
In particular, $\cH$ is finite-dimensional. 
\end{thm}

\subsection{Boundary conditions} 
\mlabel{subsec:6.3}

Let $X$ be a compact manifold with boundary $\partial X$ and let $\fK\rightarrow X$ be a Lie algebra bundle
 whose typical fiber is a compact semisimple Lie algebra $\fk$.
We then denote by $\Gamma^{k}(\fK)$  the Lie algebra of $C^k$-sections of $\fK$,
and by 
\[
\Gamma^k_0(\fK) := \{s \in \Gamma^k(\fK) \,;\, j^k (s)|_{\partial X} = 0 \} 
\]
the subalgebra of sections 
whose derivatives
of order $\leq k$ vanish at the boundary.
The negative of the Killing form 
yields a smoothly varying inner product 
$\kappa \colon \fK_x \times \fK_x \rightarrow \mathbb{R}$ on the fibers. 
If we choose a metric $g$ on $X$ and a Lie connection $\nabla$ on $\fK$, both nondegenerate at 
$\partial X$, then $\nabla$ combines with the Levi-Civita connection to 
the covariant derivative
$\nabla^n \: \Gamma(\fK) \rightarrow \Gamma(\fK \otimes (T^{*}X)^{\otimes n})$.
The scalar product $\kappa \otimes {g^{*}_{x}}^{\otimes n}$
induces the norm $\|\,\cdot\, \|_{g,x}$ on
$\fK_{x} \otimes (T^*_{x}X)^{\otimes k}$,
and we obtain 
the $C^{k}$-norm 
\[
 \|s\|_{k} := \sup \left\{ {\textstyle \sum_{j=0}^{k}} \|\nabla^j s(x) \|_{g,x}\,;\, x\in X \right\}
\]
on $\Gamma^{k}(\fK)$. This norm (multiplied by a suitable constant  
to guarantee $\|[s,s']\|_{k} \leq \|s\|_{k}\|s'\|_{k}$)
makes $\Gamma^k(\fK)$ into a Banach Lie algebra 
with
$\Gamma^{k}_0(\fK)$ as a closed ideal.
Note that different choices of $g$ and $\nabla$ yield equivalent norms.


Since the inclusion $\Gamma(\fK) \hookrightarrow \Gamma^{k}(\fK)$
is continuous and dense, every irreducible bounded unitary representation 
of $\Gamma^{k}(\fK)$ restricts to an irreducible bounded unitary representation 
of $\Gamma(\fK)$, by which it is uniquely determined.
By Theorem \ref{thm:d.24}, these are finite tensor products of irreducible
evaluation representations, hence extend to $\Gamma^{k}(\fK)$.
The Lie algebras $\Gamma(\fK)$
and $\Gamma^{k}(\fK)$ of smooth and $C^k$ sections thus have the same 
bounded irreducible representations, given by \ref{thm:d.24}.
The following theorem generalizes this observation to representations 
that need not be irreducible. 

\begin{thm} If $X$ is compact, then any bounded unitary representation of 
$\Gamma(\fK)$ extends to the Banach--Lie algebra 
$\Gamma^0(\fK)$ of continuous sections of $\fK$ and even to the 
Banach--Lie algebra $\Gamma^b(\fK)$ of all bounded sections. 
\end{thm}

\begin{prf} Let $(\pi, \cH)$ be a bounded representation 
of $\Gamma(\fK)$ and $\cA := C^*(\im(\pi))$ the $C^*$-algebra generated by 
its image. We have to show that the linear map 
$\pi \: \Gamma(\fK) \to \cA$ 
is continuous with respect to the norm 
$\|s\|_\infty := \sup_{x \in X} \|s(x)\|$. 

To this end, we recall that the irreducible representations 
of a $C^*$-algebra determine its 
norm (\cite[Thm.~2.7.3]{Dix77}). For every irreducible 
representation $\beta$ of $\cA$, the unitary representation 
$\beta \circ \pi$ of $\Gamma(\fK)$ is bounded and irreducible, 
hence a finite tensor product of evaluation representations 
$\pi_{x,\rho}$ (Theorem~\ref{thm:d.24}). 
We may therefore assume that $\pi$ is a direct sum of irreducible 
representations $\pi_{\bx_j,\rho_j}$, $j \in J$. We now have to show that 
there exists a $C > 0$ with 
\[ \|\pi(s)\| = \sup_{j \in J} \|\pi_{\bx_j, \rho_j}(s)\| \leq C \|s\|_\infty 
\quad \mbox{ for } \quad s \in \Gamma(\fK).\] 
With $\bx_j = \{ x_j^1, \ldots, x_j^{n_j}\}$ 
and $\|\pi_{\bx_j,\rho_j}\| = \sum_{i = 1}^{n_j} \|\rho_j^i\|$ 
we thus obtain 
\[ \|\pi(s)\| 
= \sup_{j \in J} \|\pi_{\bx_j, \rho_j}(s)\| 
\leq \Big(\sup_{j \in J} \sum_{i = 1}^{n_j} \|\rho_j^i\| \Big)\|s\|_\infty.\] 
It now remains to show that $C := 
\sup_{j \in J} \sum_{i = 1}^{n_j} \|\rho_j^i\|< \infty$. 
Since every summand $\pi_{\bx_j, \rho_j}$ defines a bounded representation 
of $\Gamma^0(\fK)$, it follows that $\pi$ extends to a continuous 
representation of the Lie algebra $\Gamma^b(\fK)$ 
of all bounded sections.

Let $y_0 \in X$ and pick open neighborhoods 
$V = V(y_0) \subeq U$ of $y_0$ such that $\oline V \subeq U$ is compact and 
$\fK\res_U$ is trivial. We identify 
$\Gamma_c(\fK\res_U) \subeq \Gamma(\fK)$ with 
$C^\infty_c(U,\fk)$. Let $\chi \in C^\infty_c(U,\R)$ with 
$0 \leq \chi \leq 1$ and $\chi\res_{V} = 1$. We thus obtain a linear 
embedding 
$\gamma \: \fk \to \Gamma_c(\fK\res_U) \subeq 
\Gamma(\fK), x \mapsto \chi x.$
Since $\pi$ defines a bounded representation of 
$\Gamma(\fK)$, the operator 
$\pi(\gamma(x))$ is bounded for every $x \in \fk$. In particular, we have 
\[ \sup_{j \in J} \sum_{x_j^i \in V \cap \bx_j} \|\rho_j^i(x)\| < \infty, \]
where we consider for $x_j^i \in V$ the representation $\rho_j^i$ of 
$\fK_{x_j^i}$ as a representation of $\fk$. As $\fk$ is finite-dimensional, 
using these estimates for a linear basis of $\fk$, we obtain 
\[ C_{y_0} := \sup_{j \in J} \sum_{x_j^i \in V \cap \bx_j} \|\rho_j^i\| < \infty. \] 
Since $X$ is compact, there are finitely many points 
$y_1, \ldots, y_N \in X$ with $X \subeq \bigcup_{k = 1}^N V(y_j)$. 
Then $\|\pi_{\bx_j, \rho_j}\| \leq \tilde C := \sum_k C_{y_k}$ holds for 
every $j \in J$, and we thus obtain 
$\|\pi(s)\| \leq \tilde C \|s\|_\infty$. 
\end{prf}

This shows that, for all $k \geq 0$, 
the Banach--Lie algebra $\Gamma^{k}(\fK)$ of $C^k$-sections
has the same bounded unitary representation theory as the 
Fr\'echet--Lie algebra $\Gamma(\fK)$ of smooth sections.  

For the Lie algebra $\Gamma^{k}_{0}(\fK)$, the situation is slightly more subtle, 
even in the case of irreducible representations.
Although every bounded unitary 
factor (irreducible) representation of $\Gamma^{k}_{0}(\fK)$
restricts to a bounded unitary 
factor (irreducible) representation of 
$\Gamma_{c}(\fK|_{X^{\circ}})$, the latter 
do not always extend to the former.
The following theorem shows that 
bounded unitary 
factor (irreducible) representation of 
$\Gamma_{c}(\fK|_{X^{\circ}})$
(which are all obtained by 
infinite tensor products of irreducible evaluation representations, cf. Theorem \ref{thm:4.12}), 
extend to bounded representations of $\Gamma^k_{0}(\fK)$
only if the highest weights of the evaluation representations
satisfy the following
growth condition as the points approach the boundary.

\begin{thm}\label{Growth}
Let $X$ be a compact manifold with boundary, and $\fK \rightarrow X$ a Lie algebra bundle 
whose typical fiber is a compact simple Lie algebra.
Then every factor (irreducible) bounded unitary representation
of $\Gamma^k_{0}$ is 
equivalent to
\[
\pi_{\mathbf{x},\rho,\beta} \colon \Gamma^{k}_{0}(\fK) \rightarrow B(\cH)\,,
\quad
s \mapsto \sum_{x\in \mathbf{x}} \beta \circ \iota_x  (\rho_x(s(x)))\,,
\]
where
$(\mathbf{x},\rho,\beta)$ is a triple with
$\mathbf{x} \subset X^{\circ}$ 
a locally finite subset, 
$\rho = \{(\rho_{x},V_x) \,;\, x \in \mathbf{x}\}$
a set of irreducible representations of $\fK_{x}$ 
with highest weight $\lambda_x$ satisfying
\begin{equation}\label{FastEnough}
\sum_{x\in \mathbf{x}} \|\lambda_x\|_{\kappa} \,d(x,\partial X)^k < \infty\,, 
\end{equation}
and $(\beta,\cH)$ a factor (irreducible) representation of
the $C^*$-algebra $\cA_{\mathbf{x},\rho} = 
\widehat{\bigotimes}_{x\in \mathbf{x}}B(V_x)$ (cf.\ \eqref{eq:iota}).
\end{thm}
In order to prove Theorem \ref{Growth}, we will have use for the following lemma.

\begin{lem}\label{NormRep}
Let $\fk$ be a compact simple Lie algebra with invariant scalar product $\kappa$. 
Then there exists a constant ${C(\fk) > }0$ such that 
\begin{equation}\label{NormRepEst}
 C(\fk) \|\lambda\|_{\kappa} \|x\|_{\kappa} \leq \|\rho(x)\| \leq \|\lambda\|_{\kappa} \|x\|_{\kappa} 
\end{equation}
holds for all
irreducible representations $\rho$ with highest weight $\lambda$, and for all $x \in \fk$.
\end{lem}

\begin{prf} Let $\ft \subeq \fk$ be maximal abelian and 
$\fh := \ft_\C \subeq \g := \fk_\C$ the corresponding 
Cartan subalgebra (cf.\ Section~\ref{sec:2}). We also fix a connected Lie group $K$ 
with Lie algebra $\fk$. We denote by $\cW := N_K(\ft)/Z_K(\ft)$ the Weyl group 
of $(\fk,\ft)$, and by $\Delta^+ \subeq \Delta \subeq \fh^*$ a positive 
system with respect to which $\lambda$ is the highest weight of $\rho$. 
Recall that every adjoint orbit $\Ad(K)x$ intersects 
\[ \ft_+ := \{ z \in \ft \: (\forall \alpha \in \Delta^+)\, i\alpha(z) \geq 0\},\] 
so we may assume that $x \in \ft_+$.

First we recall that the set $\cP_\rho$ of $\ft$-weights of $\rho$ 
is $\cW$-invariant, contains $\lambda$ and is contained in 
$\conv(\cW\lambda)$ (cf.\ \cite[Ch.~VIII]{Bou90}). 
As $\cW$ acts isometrically on $\ft^*$ with respect to the induced 
norm, the relation 
\[ \|\rho(x)\|= \sup |\la \cW \lambda, x \ra| \leq \|\lambda\|_\kappa \|x\|_\kappa\] 
for $x \in \ft$ follows from the Cauchy--Schwarz inequality. 

It remains to show that there exists a constant $C > 0$ with  
$\|\rho(x)\| \geq C \|\lambda\|_\kappa$ for every $x \in \ft_+$ 
with $\|x\| = 1$. 
For $\beta \in i\ft^*$, let 
$t_\beta \in \ft$ be the unique element with $\kappa(t_\beta, z) = i\beta(z)$ for 
$z \in \ft$. Then 
$\ft_+ = \{ z \in \ft \: (\forall \alpha \in \Delta^+)\, \kappa(t_\alpha, z) 
\geq 0\}$.  
For $\alpha \in \Delta$, we have 
$-i\check \alpha = \frac{2}{\|t_\alpha\|^2} t_\alpha$, so that 
$\kappa(t_\lambda, t_\alpha) = i \lambda(t_\alpha) \geq 0$ for $\alpha \in \Delta^+$ 
implies $t_\lambda \in \ft_+$. 

Now we observe that 
\[ \|\rho(x)\|\geq |\lambda(x)| = \kappa(t_\lambda, x),\] 
so that it remains to show that 
\[ \inf \{ \kappa(y,z) \: \|y\| = \|z\| = 1, y,z \in \ft_+ \} > 0.\] 
By compactness, it suffices to show that 
$\kappa(y,z) > 0$ for $0 \not= y,z \in \ft_+$. We argue by contradiction 
and assume that this is not the case. Then there exists non-zero 
$y,z \in \ft_+$ with $\kappa(y,z) \leq 0$. 
Since 
\[ \cW z \subeq z - \sum_{\alpha \in \Delta^+} \R_+ t_\alpha \]
by \cite[Prop.~V.2.7(ii)]{Ne00}, we obtain 
$\kappa(y,z') \leq  0$ for $z' \in \conv(\cW z)$. 
The simplicity of $\fk$ implies that 
$\ft$ is a simple $\cW$-module, so  
$f := \sum_{w \in \cW} w z$ is an element of $\ft^\cW= \{0\}$. 
Since $\cW z$ spans $\ft$, there exists a $w \in \cW$ 
with $\kappa(w z, y) \not=0$, hence $< 0$, but this leads to the contradiction 
$0 = \kappa(y,f) < 0$. 
\end{prf}

\begin{prf}{\it (of Theorem \ref{Growth})}
Since $X$ is compact, the continuous inclusion
$\Gamma_c(\fK|_{X^{\circ}}) \hookrightarrow \Gamma^k_0(\fK)$
of Lie algebras is dense.
Thus
every norm continuous factor (irreducible) representation of $\Gamma^k_0(\fK)$ 
restricts to a 
norm continuous factor (irreducible)
representation of $\Gamma_c(\fK|_{X^{\circ}})$, and is uniquely defined by 
this restriction. Theorem \ref{thm:4.12}
then yields a locally finite set $\mathbf{x} \subseteq X^{\circ}$ and representations 
$\rho_{x}$ of $\fK_x$ and $\beta$ of $\cA_{\mathbf{x},\rho}$
such that $\pi|_{\Gamma_{c}(\fK|_{X^{\circ}})} = \beta \circ \eta_{\mathbf{x},\rho}$.
Since $\beta$ is an isometry, 
the representation of $\Gamma_{c}(\fK|_{X^{\circ}})$ extends to a bounded unitary 
representation of $\Gamma^k_{0}$ if and only if 
$\eta_{\mathbf{x},\rho} \colon \Gamma_{c}(\fK|_{X^{\circ}}) \rightarrow \cA_{\mathbf{x},\rho}$
is continuous w.r.t.\ the $C^k$-norm.
In the following, we write 
$d_x := d(x,\partial X)$ for brevity.

We first prove that $\sum_{x\in \mathbf{x}} \|\lambda_x\|_{\kappa} d^k_x < \infty$
implies continuity of $\eta_{\mathbf{x},\rho}$.
Recall that
$\eta_{\mathbf{x},\rho}(s) := \sum_{x\in \mathbf{x}} \iota_x \rho_x(s(x))$, where 
$\iota_x \colon B(V_x) \rightarrow \cA_{\mathbf{x},\rho}$
is the canonical inclusion.
Since the $\iota_x$ are isometries, we have
\begin{equation}\label{NormDef}
\|\eta_{\mathbf{x},\rho}(s)\| = \sum_{x\in \mathbf{x}} \| \rho_x(s(x)) \|\,.
\end{equation}
By Lemma \ref{NormRep}, we have the estimate $\|\rho_x(s(x))\| \leq 
\|\lambda_x\|_{\kappa} \|s(x)\|_{\kappa}$.
Since every Lie connection is orthogonal w.r.t.\ $\kappa$, we can combine
Taylor's Theorem with the parallel transport 
equation to yield $\|s(x)\|_{\kappa} \leq \frac{1}{k!} d_x^k \|s\|_k$, 
and thus $\|\rho_x(s(x))\| \leq \frac{1}{k!}\|\lambda_x\|_{\kappa} d_x^k \,\|s\|_{k}$.
Now (\ref{NormDef}) yields
$\|\eta_{\mathbf{x},\rho}(s)\| \leq \frac{1}{k!}(\sum_{x\in \mathbf{x}} \|\lambda_x\|_{\kappa} d^k_x)
\|s\|_k $, which shows that $\eta_{\mathbf{x},\rho}$ is continuous
if $\sum_{x\in \mathbf{x}} \|\lambda_x\|_{\kappa} d^k_x < \infty$.

Conversely, suppose that $\sum_{x\in \mathbf{x}} \|\lambda_x\|_{\kappa} d^k_x = \infty$.
We will exhibit sections $s_{\gamma}$ with $\|\eta_{\mathbf{x},\rho}(s_{\gamma})\| \geq 1$ but
$\|s_{\gamma}\|_{k}$ arbitrarily small as $\gamma \downarrow 0$, showing that
$\eta_{\mathbf{x},\rho}$ is not continuous in the $C^k$-norm.

The first step is to localize the problem.
Using the exponential flow for the \mbox{metric $g$},
we find an $\varepsilon \in (0,1/2)$,
a finite covering of $\partial M$ by $U_i \subseteq \partial M$,
open neighborhoods $V_i \supset \overline{U}_{i}$, 
and local diffeomorphisms 
$\phi_i \colon V_{i} \times [0,2 \varepsilon) \rightarrow X$,
such that $\fK$ trivializes over $\im(\phi_i)$
and such that 
$d(\phi_i(v,x),\partial X) = x$ (cf. \cite[Ch.~9, Thm.~20/21]{Sp70}).
The local diffeomorphisms are defined by the flow
$\phi_i(v,t) = \exp_{v}(t\vec{n}_{v})$
along a (locally defined) inward pointing normal vector field $\vec{n}$ on $\partial M$.  

Since $\mathbf{x}$ is locally finite in $X^{\circ}$, and
since $X^{\circ}$ is covered by the open sets
$\hat{U}_i := \phi_i(U_i \times (0,\varepsilon))$ together with 
the compact set $\{x\in M \,;\, d_x \geq \varepsilon/2\} \subset X^{\circ}$, 
there is at least one $i$ such that
$\sum_{x \in \mathbf{x} \cap \hat{U}_i} \|\lambda_x\|_{\kappa} d_x^k= \infty$.

Now let $\delta > 0$ and let $\xi \in C^{\infty}_{c}(\im(\phi_i),\fk)$ 
be such that its restriction to $\hat{U}_i$ is 
a constant $\xi_0 \in \fk$ with $\|\xi_0\|_{\kappa} = 1$.
Using the trivialization of $\fK$ over $\im(\phi_i)$,
we define for every $\gamma \in (0,1)$ the section 
$s_{\gamma} \in \Gamma^k_{0}(X)$
by 
$s_{\gamma}( \phi_i (u,x)) := 
\delta \,x^{k + \gamma} \xi$. 
Since the topology induced by the $C^k$-norm is independent of the choice of
connection, we may as well choose $\nabla$ to be compatible
with the trivialization over $\mathrm{supp}(\xi)$.
We then have $\|s_{\gamma}\|_{k} \leq \delta C \|\xi\|_{k} \|x^{k+\gamma}\|_{k}$
for some suitable constant $C$, with 
\[\textstyle
\|x^{k+\gamma}\|_{k} := \sup\left\{\sum_{j=0}^{k} |\frac{d^j}{dx^j} x^{k+\gamma}|\:x\in [0, 2\varepsilon]\right\}
\leq \sum_{j=0}^{k+1}\frac{(k+1)!}{j!} \leq e\,(k+1)!\,.
\]
Thus $\lim_{\delta \downarrow 0} \|s_{\gamma}\|_{k} = 0$ uniformly in $\gamma$.   
However, we have $\|\rho_x(s(x))\| \geq C(\fk) \|\lambda_x\|_{\kappa} \|s(x)\|_{\kappa}$
by Lemma \ref{NormRep},
so that 
\begin{equation}\label{CriticalSection}
\|\eta_{\mathbf{x},\rho}(s_{\gamma})\| \geq 
{\delta \, C(\fk)} \sum_{x \in \mathbf{x} \cap \hat{U}_{i}} 
\|\lambda_x\|_{\kappa} d_x^{k+\gamma}\,.
\end{equation}
Now since 
$\sum_{x \in \mathbf{x} \cap \hat{U}_{i}} \|\lambda_x\|_{\kappa} d_x^k= \infty$,
there exists for every $N \in \mathbb{N}$ a finite
subset $\mathbf{x}_0 \subset \mathbf{x} \cap \hat{U}_{i}$
such that 
$\sum_{x \in \mathbf{x}_0 } \|\lambda_x\|_{\kappa} d_x^k > N$.
Taking the limit $\gamma \rightarrow 0$ in the inequality
$\|\eta_{\mathbf{x},\rho}(s)\| \geq 
\delta \, C(\fk) \sum_{x \in \mathbf{x}_0} 
\|\lambda_x\|_{\kappa} d_x^{k+\gamma}$ 
and choosing $N>1/(\delta \, C(\fk))$, 
we see that there exists a $\gamma > 0$
for which $\|\eta_{\mathbf{x},\rho}(s_{\gamma})\| > 1$.
Since  
$\lim_{\delta\downarrow 0} \|s_{\gamma}\|_{k} = 0$
uniformly in $\gamma$, this finishes the proof.\end{prf}

The following is a direct consequence of Theorem \ref{Growth} in the case $k=0$.
\begin{cor}
Every bounded irreducible unitary representation $(\pi,\mathcal{H})$ of $\Gamma^0_0(\fK)$ 
is a finite tensor product of irreducible evaluation representations. 
In particular, $\mathcal{H}$ is
finite-dimensional. 
\end{cor}

\subsection{Further problems}

It is an interesting question to which extent the 
theory developed in Sections~\ref{sec:2} and \ref{sec:3} can be 
extended to the case where $\cA$ is a general non-unital cia.
Not unrelated, it would also be interesting to see whether 
the results on Lie algebras of sections in Section~\ref{sec:4}
extend to the context of complex analytic or
algebraic geometry.

\subsubsection{Unital cias}

It is straightforward to apply our techniques to the \emph{unital}
cia $\cA = \cO_{\mathrm{an}}(T^n)$ of analytic functions 
on the torus $T^n \subset \C^n$ (cf.\ Example~\ref{ex:cia}~(d)).
The characters are given by evaluation in points of $T^n$, and are thus
in particular involutive w.r.t.\ the
involution $f^*(z_1,\ldots, z_n) := \overline{f}(1/\overline{z}_1,\ldots,1/\overline{z}_n)$.
Corollary~\ref{cor:4.9} 
guarantees that every holomorphic multiplicative map $\phi \: \cO_{\mathrm{an}}(T^n) \rightarrow \C$ 
is given by evaluation in finitely many points $p_i \in T^n$, i.e.,  
$\phi(f) = \Pi_{i=1}^{N}f(p_i)$.
Using Theorem~\ref{thm:tenspro}, we then see that every 
irreducible bounded $*$-representation of the Lie algebra $\cO_{\mathrm{an}}(T^n,\fg)$
of analytic functions on $T^n$ with values in $\fg = \fk_{\C}$, for 
$\fk$ a compact semisimple
Lie algebra, is a tensor product of finitely many irreducible
evaluation representations in different points $p_i$ of $T^n$.

\subsubsection{Non-unital cias}

If $\cA$ is a non-unital cia, then the results of 
Sections~\ref{sec:2} and \ref{sec:3} do not apply.
The localization techniques of Section~\ref{sec:4}
will work for certain
classes of cias which are soft sheaves over their spectra, 
but more `rigid' non-unital cias, such as the algebra
$\cA_p = \cO_{\mathrm{an}}(T^n)_{p}$ of analytic functions on $T^n$ that
vanish at $p\in T^n$, cannot be handled in this way.

It is clear that the Lie algebra $\fg(\cA_p)$
allows for infinite tensor products of evaluation representations, 
coming from 
holomorphic multiplicative involutive maps of the form
\[ F \: \1 + \cA_p \to \C, \quad 
(1 + f) \mapsto \prod_{m =1}^{\infty} (1 + f(p_m))\,, \] 
where $(p_m)_{m\in \N}$ is
a sequence of points in $T^n$ with 
$\sum_{m=0}^{\infty} d(p_m,p) < \infty$
(cf.\ Examples
\ref{ex:6.3} and \ref{ex:6.4}). Here we use that, for any 
bounded set $S$ of holomorphic functions in a neighborhood $U$ of $T^n$, 
the set $\{f'(p) \: f \in S\}$ of derivatives in $p$ is bounded 
by the Cauchy estimates. 

In this connection, the following question arises: 
Suppose that $\cA$ is a non-unital involutive cia, and 
$F \: \1 + \cA \to \C$ a holomorphic, multiplicative
and involutive map.
Is the functional $\lambda := \dd F(\1) \: \cA \to \C$ inducible? 
As the following proposition shows, it makes a serious difference whether 
we consider multiplicative homomorphisms on $(\cA,\cdot)$ or on $(\1 + \cA, \cdot)$. 

\begin{prop} Let $\cA$ be a non-unital commutative $\bk$-algebra. 
Then every multiplicative map 
$\chi \: (\cA, \cdot)\to \C$ 
extends to a multiplicative map 
$\tilde\chi \: (\cA_+, \cdot)\to \C$ with 
$\tilde\chi(\1) = 1$. 
If, in addition, $\cA$ is a complex cia and $\chi$ is holomorphic, 
then the same holds for $\tilde\chi$ and then $\chi$ is a finite 
product of algebra homomorphisms.   
\end{prop}

\begin{prf} If $\chi = 0$, then we put 
$\chi(a + t \1) := t$. We now assume that $\chi \not=0$. 
For $a \in \cA_+$, we write 
$\lambda_a \: \cA \to \cA$ for the multiplication with $a$. 
Then there exists a $b\in \cA$ with $\chi(b) \not=0$. We put 
\[ \chi_b(a) := \frac{\chi(ab)}{\chi(b)}\quad \mbox{ for } \quad a \in \cA_+, b \in \cA.\] 
For any other $c \in \cA$ with $\chi(c) \not=0$ and $a \in \cA_+$, we then 
have 
\[ \chi_b(a) \chi(b) \chi(c) 
= \chi(ab) \chi(c) = \chi(abc) 
= \chi(acb) = \chi(ac)\chi(b) = \chi_c(a) \chi(c)\chi(b),\] 
so that $\chi_b(a) = \chi_c(a)$. 
Therefore 
\[ \tilde\chi \: \cA_+ \to \C, \quad 
\tilde\chi(a) := \chi_b(a) \quad \mbox{ for } \quad 
\chi(b) \not=0 \] 
is a well-defined extension of~$\chi$. 
It satisfies $\chi(ab) = \tilde\chi(a) \chi(b)$ for $\chi(b) \not=0$. 
If $\chi(b) = 0$, then we choose a $c \in \cA$ with $\chi(c) \not=0$ and 
obtain 
\[ \chi(ab)\chi(c) =  \chi(abc) = \chi(bac) = \chi(b) \chi(ac) = 0,\] 
so that $\chi(ab) =0$. This shows that 
\[ \chi(ab) = \tilde\chi(a)\chi(b) \quad \mbox{ for } \quad 
a \in \cA_+, b \in \cA.\] 
For $a,a' \in \cA_+$ and $\chi(b) \not=0$, we then have 
\[ \tilde\chi(aa') = \chi_b(aa') 
= \frac{\chi(aa'b)}{\chi(b)}
= \tilde\chi(a)\frac{\chi(a'b)}{\chi(b)}
= \tilde\chi(a)\tilde\chi(a').\]
Therefore $\tilde\chi$ is multiplicative. 

Now suppose that $\cA$ is a complex cia and that 
$\chi$ is holomorphic. For any $b \in \cA$ with $\chi(b) \not=0$, 
the relation $\tilde\chi(a) = \chi_b(a)$, together with 
the holomorphy of the map $\cA_+ \to \cA, a \mapsto ab$ 
now implies that also $\tilde\chi$ is holomorphic. 
We conclude with Theorem~\ref{thm:2.5} that $\tilde\chi$ 
is a finite product of characters. 
\end{prf}

Not all bounded irreducible $*$-representations of $\cA_p = \cO_{\mathrm{an}}(T^n)_{p}$
are highest weight representations.
The Lie algebra
$\fsl_2(\cA_p)$ has an 
irreducible bounded \mbox{$*$-representation}, 
\[ \pi(X) := \bigoplus_{m\in\N} \,\1^{m} \otimes X(p_m) \otimes \1^{\infty},\] 
on any infinite tensor product space 
$\cH := \bigotimes_{m \in \N} (\C^2, v_m),$
where $v_m \in \C^2$ is a sequence of unit vectors and $(p_m)_{m\in \N}$
a sequence of points in $T^n$ with $\sum_{m\in \N}d(p_m,p) < \infty$
(cf.\ \eqref{eq:inftens} and Example~\ref{ex:a.5}). 
In particular, we have irreducible bounded unitary representations 
of $\fsl_2(\cA_p)$ 
where $\cH^{\g^-} = \{0\}$. 

Clearly, these cannot be classified 
with the highest weight theory from Section~\ref{sec:2}. 
This raises the following question:
Suppose that $\cE = \cH^{\g^-}$ is non-zero for an irreducible bounded unitary 
representation of $\fk_{\cA}$, with $\cA$ a non-unital involutive cia. 
Then the arguments from the proof of Proposition~\ref{prop:6.1} 
still imply that $\dim \cE = 1$, so that the representation 
$(\rho, \cE)$ of $\g^0 = h \otimes \cA$ is given by a linear functional 
$\lambda \: \cA \to \C$. In view of the above example, $\lambda$ 
can be an infinite sum $\sum_{j=1}^\infty \lambda_j \otimes \chi_j$, 
so that Theorem~\ref{thm:ind-charac} does not extend to the non-unital case. 
Does $F(e^a) := e^{\lambda(a)}$ extend to a multiplicative holomorphic map 
$F \: \cA \to \C$? 

\subsubsection{Multiloop algebras}

It would also be interesting to see whether 
the results on Lie algebras of sections in Section~\ref{sec:4}
extend to the context of complex analytic or
algebraic geometry. 

The class of \emph{twisted multiloop algebras} are an interesting example.
In this case, 
the finite group $H = \Pi_{i=1}^{n}\Z/k_i \Z$
acts on $\fg$ by automorphisms, and 
on $T^n$ by the free action 
$h \cdot z := (e^{2\pi i h_1/k_1}z_1,\ldots,e^{2\pi i h_n/k_n}z_n)$.
Then $\fK := T^n \times_{H} \fk$ is a bundle of Lie algebras over 
$T^n/H \simeq T^n$, and the twisted multiloop algebras
in the algebraic, analytic and smooth setting are the corresponding Lie algebras 
of sections
\begin{eqnarray}\label{eqn:kinds}
\cO_{\rm alg}(T^n,\fg)^H \subset \cO_{\mathrm{an}}(T^n,\fg)^H \subset C^{\infty}(T^n,\fg)^H\,,
\end{eqnarray}
where $\cO_{\rm alg}(T^n,\fg) := \C[t_1^{\pm},\ldots,t_n^{\pm}] \otimes \fg$.
Assuming that the action of $H$ on $\fg$ preserves a compact real form $\fk$,
Theorem~\ref{thm:d.24} implies that every irreducible bounded $*$-representation
of $\Gamma(\fK)_{\C} \simeq C^{\infty}(T^n,\fg)^{H}$
is a finite tensor product of irreducible evaluation representations
in different points $p_i$ of $T^n$.
Analogously, it follows from \cite{La10} that every
finite-dimensional irreducible representation of 
$\cO_{\rm alg}(T^n,\fg)^H$ is given by a finite tensor product of evaluation representations,
which implies the corresponding result in the analytic and smooth setting because the inclusions
in (\ref{eqn:kinds}) are dense. 
(See also \cite{NSS12}, where the action of $H$ on $T^n$ is generalized to the 
automorphic action of a finite group on a scheme.)

This striking similarity raises the question whether our results extend to the 
analytic/algebraic setting also if the bundle $\fK \rightarrow X$ does not 
allow for a flat Lie connection, so that the structure group is not 
finite.


\begin{thebibliography}{aaaaaaa} 

\bibitem[AH78]{AH78} Albeverio, S., and R. J. H\o{}egh-Krohn, 
{\it The energy representation of Sobolev--Lie groups}, 
Compositio Math. {\bf 36:1} (1978), 37--51 

\bibitem[A-T93]{A-T93} Albeverio, S., R. J. H\o{}egh-Krohn, J. A. Marion,  
D. H. Testard, and B. S. Torresani, ``Noncommutative Distributions
-- Unitary representations of Gauge Groups and Algebras,'' Pure and
Applied Mathematics {\bf 175}, Marcel Dekker, New York, 1993 

\bibitem[An10]{An10}
Ando, H., {\it On the local structure of the representation of a
local gauge group},
Infin. Dimens. Anal. Quantum Probab. Relat. Top. {\bf 13:2} (2010), 223–-242


\bibitem[Ar69]{Ar69} Araki, H., {\it 
Factorizable representation of current algebra-- 
Non commutative extension of the L\'evy-Kinchin
formula and cohomology of a solvable group
with values in a Hilbert Space}, Publ. RIMS {\bf 5:3} (1969), 361--422 



\bibitem[Are47]{Are47} Arens, R., {\it Linear topological division algebras}, 
Bull. Amer. Math. Soc. {\bf 53} (1947),  623--630 

\bibitem[Be79]{Be79} Berezin, F. A., {\it 
Representations of the continuous direct product of universal coverings 
of the group of motions of the complex ball}, 
Trans. Moscow Math. Soc. {\bf 2} (1979), 281--289 


\bibitem[Bi07]{Bi07} Biller, H., {\it Analyticity and naturality of the multi-variable 
functional calculus}, Expo. Math. {\bf 25} (2007), 131--163

\bibitem[Bou90]{Bou90} N. Bourbaki, ``Groupes et alg\`ebres de Lie, Chapitres VII-VIII,'' 
Masson, Paris, 1990

\bibitem[CP86]{CP86} Chari, V., and A. Pressley, 
{\it New unitary representations 
of loop groups}, Math. Ann. {\bf 275} (1986), 87--104 

\bibitem[CP87]{CP87} ---, {\it Unitary representations of the maps $S^1 \to su(N,1)$}, 
Math. Proc. Camb. Phil. Soc. {\bf 102} (1987), 259--272 

\bibitem[Dix77]{Dix77}
Dixmier, J., ``$C^*$-algebras,''
North Holland Publishing Company, Amsterdam,  New York, Oxford, 1977

\bibitem[Dix96]{Dix96} ---, {\it Les alg\`ebres d'op\'erateurs dans 
l'espace Hilbertien}, \'Editions Gabay, 1996. 


\bibitem[GK03]{GK03} Galino, E., and A. Kaplan, {\it 
On unitary representations of nilpotent gauge groups}, Comm. Math. Phys. {\bf 236} 
(2003), 187--198 

\bibitem[GG68]{GG68} Gelfand, I.M., and M.I.~Graev, {\it 
Representations of quaternion groups over locally compact fields and 
function fields}, Funk. Anal. Prilo\v zen {\bf 2:1} (1968), 20--35

\bibitem[Gli60]{Gli60} Glimm, J., 
{\it On a certain class of operator algebras}, Transactions of the 
Amer. Math. Soc. {\bf 40} (1960), 318--340 


\bibitem[Gl02]{Gl02} Gl\"ockner, H., 
{\it Algebras whose groups of units are Lie groups}, 
Studia Math. {\bf 2} (2002), 147--177 

\bibitem[Gl04]{Gl04} Gl\"ockner, H., 
{\it Lie groups over non-discrete topological fields},
arXiv:math/0408008 [math.GR], 2004.



\bibitem[GN13]{GN13} Gl\"ockner, H., and K.-H. Neeb, ``Infinite dimensional 
Lie groups, Vol. I, Basic Theory and Main Examples,'' book in preparation 

\bibitem[Gui72]{Gui72} Guichardet, A., ``Symmetric Hilbert Spaces and 
Related Topics,'' Lecture Notes in Math. {\bf 261}, Springer, 1972 

\bibitem[GGV80]{GGV80} Gelfand, I. M., Graev, M. I., and A. M. Vershik, 
{\it Representations of the group of functions taking values in a compact Lie group}, 
Compositio Math. {\bf 42:2} (1980), 217--243 

\bibitem[He82]{He82}  Henrichs, R. W., {\it  
Decomposition of invariant states and nonseparable $C^*$-algebras}, 
Publ. R.I.M.S. Kyoto Univ. {\bf 18:1} (1982), 159--181 

 

\bibitem[Is76]{Is76} Ismagilov, R. S., {\it Unitary representations 
of the group $C^\infty_c(X,G)$, $G = \SU_2$}, 
 Mat. Sb. (N.S.) {\bf 100(142):1(5)} (1976), 117--131, 166 

\bibitem[Is96]{Is96} ---, ``Representations of 
Infinite-Dimensional Groups'', Translat. of Math. Monographs {\bf 152}, 
Amer. Math. Soc., 1996 

\bibitem[JK85]{JK85} Jakobsen, H.~P., and V.\ Kac, 
{\it A new class of unitarizable
highest weight representations of infinite-di\-men\-sio\-nal Lie algebras}, 
in ``Non--linear equations in classical and quantum field theory,''
N.\ Sanchez ed., Springer Verlag, Berlin,
Heidelberg, New York, Lecture Notes in Physics {\bf 226} (1985), 1--20

\bibitem[JK89]{JK89} ---, {\it A new class of unitarizable
highest weight representationsof infinite-di\-men\-sio\-nal Lie algebras,
II}, J.\ Funct.\ Anal. {\bf 82} (1989), 69--90 


\bibitem[JN13]{JN13} Janssens, B., and K.-H. Neeb, {\it 
Projective unitary representations extending bounded factor representations 
of mapping Lie algebras}, in preparation 

\bibitem[JW10]{JW10} Janssens, B., and C. Wockel, {\it Universal Central Extensions 
of Gauge Algebras and Groups}, J. Reine Angew. Math., to appear; 
arXiv:1010.3569 [math.DG] 


\bibitem[La10]{La10} Lau, M., {\it Representations of multiloop algebras},
Pacific J. Math., {\bf 245:1} (2010), 167–-184

\bibitem[Ma02]{Ma02} Maier, P., {\it Central extensions of topological
current algebras}, in ``Geometry and Analysis on Finite-
and Infinite-Dimensional Lie Groups,'' A.~Strasburger et al Eds., 
Banach Center Publications {\bf 55}, Warszawa 2002; 61--76 

\bibitem[Ne00]{Ne00} Neeb, K.-H., ``Holomorphy and Convexity in Lie Theory'', 
Expositions in Mathematics {\bf 28}, de Gruyter Verlag, Berlin, 2000 

\bibitem[Ne06]{Ne06} ---, {\it Towards a Lie theory of locally convex 
groups}, Jap. J. Math. 3rd ser. {\bf 1:2} (2006), 291--468 

\bibitem[Ne09]{Ne09} ---, 
{\it Semibounded unitary representations of infinite dimensional Lie groups}, 
in ``Infinite Dimensional Harmonic Analysis IV'', 
Eds. J.\ Hilgert et al, World Scientific, 2009; 209--222


\bibitem[Ne12]{Ne12} ---, {\it Positive energy representations and 
 continuity of projective representations for general topological groups}, 
arXiv:1208.2511 [math.RT], 2012  

\bibitem[NS11]{NS11} Neeb,  K.-H., and H. Sepp\"anen, {\it 
Borel--Weil Theory for Groups over Commutative Banach Algebras}, 
J. reine angew. Math. {\bf 655} (2011), 165–187 



\bibitem[NS12]{NS12} Neher, E., and A. Savage, {\it 
A survey of equivariant map algebras with open problems}, 
arXiv:1211.1024 [math.RT], 2012

\bibitem[NSS12]{NSS12} Neher, E., Savage, A., and P. Senesi, 
{\it Irreducible finite-dimensional representations of equivariant 
map algebras},  Trans. Amer. Math. Soc. {\bf 364:5} (2012), 2619--2646 

\bibitem[Ols82]{Ols82} Olshanski, G.I., {\it 
Spherical functions and characters of the group $G^X$}, Russian Math. 
Surveys {\bf 37} (1982), 233-234

\bibitem[Po67]{Po67} Powers, R. T., {\it Representations of uniformly 
hyperfinite algebras and their associated von Neumann algebras}, 
Annals of Math. {\bf 86} (1967), 138--171 

\bibitem[PS86]{PS86} Pressley, A., and G. Segal, ``Loop Groups," Oxford University Press, 
Oxford, 1986

\bibitem[Ru86]{Ru86} Rudin, W., ``Real and Complex Analysis,'' McGraw Hill, 1986 

\bibitem[Ru91]{Ru91} ---, ``Functional Analysis,'' 
{International Series in Pure and Applied Mathematics},
McGraw-Hill Inc., New York, 1991
 
\bibitem[Sa67]{Sa67} 
Sakai, S., {\it On type $I$ $C^*$-algebras}, 
Proc. of the Amer.\ Math.\ Soc. {\bf 18} (1967), 861--863

\bibitem[Sa71]{Sa71} ---, ``$C^*$-algebras and $W^*$-algebras,'' Ergebnisse der 
Math.\ und ihrer Grenzgebiete {\bf 60}, 
Springer-Verlag, Berlin,
Heidelberg, New York, 1971 

\bibitem[Sam91]{Sam91} Samoilenko, Y. S., ``Spectral Theory of 
Families of Self-Adjoint Operators'', Mathematics and its Applications 
(Soviet Series) Kluwer Acad. Publ., 1991 


\bibitem[Sp70]{Sp70} Spivak, M., ``A Comprehensive Introduction to Differential Geometry'',
Vol.~1,
Publish or Perish, Inc.\,, Boston, Mass., 1970 


\bibitem[Ta03]{Ta03}
Takesaki, M., ``Theory of Operator Algebras. II,''
Encyclopedia of Mathematical Sciences {\bf 125},
Operator Algebras and Non-commutative Geometry {\bf 6},
Springer-Verlag, Berlin, 2003

\bibitem[Th08]{Th08} Thiemann, Th., ``Modern Canonical Quantum General Relativity,'' 
Cambridge Monographs on Mathematical Physics, 2008 

\bibitem[To87]{To87} Torresani, B. S., {\it Unitary positive energy 
representations of the gauge group}, 
Letters in Math. Physics {\bf 13} (1987), 7--15 

\bibitem[Tr67]{Tr67} Treves, F., ``Topological Vector Spaces, Distributions and
Kernels,'' Academic Press, San Diego, 1967 

\bibitem[Tu70]{Tu70} Turpin, Ph., {\it Une remarque sur les alg\`ebres \`a inverse continue}, 
C. R. Acad. Sci. Paris S\'er. A-B {\bf 270} (1970), A1686-A1689

\bibitem[VGG74]{VGG74} Vershik, A. M., Gelfand, I. M., and 
M. I. Graev, {\it Irreducible representations of the group 
$G^X$ and cohomology}, Funct. Anal. and its Appl. {\bf 8:2} (1974), 67--69

\bibitem[Wae54]{Wae54} Waelbroeck, L., {\it Le calcul symbolique dans les alg\`{e}bres commutatives},
J. Math. Pures Appl. {\bf 33:9} (1954), 147--186 

\bibitem[Wag11]{Wag11} Wagner, S., {\it Extending characters on fixed point 
algebras}, arXiv:1111.5560 [math.DS], 2011

\bibitem[Wag11b]{Wag11b} ---, ``A Geometric Approach to 
Noncommutative Principal Torus Bundles,'' 
Ph.D. thesis, FAU Erlangen-Nuremberg, 2011; 
arXiv:1108.4294 [math.DG]  

\bibitem[Wo06]{Wo06} Wockel, C., {\it Smooth extensions and spaces 
of smooth and holomorphic mappings}, J. Geom. Symmetry Phys. {\bf 5} (2006) 
126--134; see also arXiv:0511064 [math.DG] 

\end{thebibliography}
\end{document}